\documentclass[11pt]{article}
\usepackage[latin1]{inputenc}
\usepackage{amsmath,amsthm,amssymb}
\usepackage{amsfonts}
\usepackage{amsmath,amsthm,amssymb,amscd}
\usepackage{latexsym}
\usepackage{color}
\usepackage{graphicx}
\usepackage{mathrsfs}
\usepackage{cite}

\usepackage{color,enumitem,graphicx}
\usepackage[colorlinks=true,urlcolor=black,
citecolor=black,linkcolor=black,linktocpage,pdfpagelabels,
bookmarksnumbered,bookmarksopen]{hyperref}

\makeatletter \@addtoreset{equation}{section} \makeatother

\newtheorem{theorem}{Theorem}[section]
\newtheorem{definition}{Definition}[section]
\newtheorem{proposition}{Proposition}[section]
\newtheorem{lemma}{Lemma}[section]
\newtheorem{remark}{Remark}[section]

\allowdisplaybreaks

\begin{document}
\title{Normalized ground states for a biharmonic Choquard equation with exponential critical growth}

\author{Wenjing Chen\footnote{Corresponding author.}\ \footnote{E-mail address:\, {\tt wjchen@swu.edu.cn} (W. Chen), {\tt zxwangmath@163.com} (Z. Wang).}\  \ and Zexi Wang\\
\footnotesize  School of Mathematics and Statistics, Southwest University,
Chongqing, 400715, P.R. China}

\date{ }
\maketitle

\begin{abstract}
{In this paper, we consider the normalized ground state solution for the following biharmonic Choquard type problem
\begin{align*}
  \begin{split}
  \left\{
  \begin{array}{ll}
   \Delta^2u-\beta\Delta u=\lambda u+(I_\mu*F(u))f(u),
    \quad\mbox{in}\ \ \mathbb{R}^4,  \\
    \displaystyle\int_{\mathbb{R}^4}|u|^2dx=c^2,\quad u\in H^2(\mathbb{R}^4),\\
    \end{array}
    \right.
  \end{split}
  \end{align*}
where $\beta\geq0$, $c>0$, $\lambda\in \mathbb{R}$, $I_\mu=\frac{1}{|x|^\mu}$ with $\mu\in (0,4)$, $F(u)$ is the primitive function of $f(u)$, and $f$ is a continuous function with exponential critical growth in the sense of the Adams inequality. By using a minimax principle based on the homotopy stable family, we obtain that the above problem admits at least one ground state normalized solution.}



\end{abstract}

\section{Introduction and statement of main result}

In this article, we are devoted to the following biharmonic Choquard type equation
\begin{equation}\label{abs1}
   \Delta^2u-\beta\Delta u=\lambda u+(I_\mu*F(u))f(u),
    \quad\mbox{in}\ \ \mathbb{R}^4,
\end{equation}
with prescribed mass
\begin{equation}\label{abs2}
  \displaystyle\int_{\mathbb{R}^4}|u|^2dx=c^2,\quad u\in H^2(\mathbb{R}^4),
\end{equation}
where $\Delta^2$ denotes the biharmonic operator, $\beta\geq0$, $c>0$, $\lambda\in \mathbb{R}$, $I_\mu=\frac{1}{|x|^\mu}$ with $\mu\in (0,4)$, $F(u)$ is the primitive function of $f(u)$, which satisfies some suitable conditions that will be specified later.

The study of \eqref{abs1} originates from seeking the standing wave solutions of the following time-dependent nonlinear Schr\"{o}dinger equation with the mixed dispersion and a general Choquard nonlinear term
\begin{equation}\label{back}
  i\psi_t-\Delta^2\psi+\beta\Delta \psi+(I_\mu*F(\psi))f(\psi)=0, \quad\psi(0,x)=\psi_0(x), \quad \psi(t,x):\mathbb{R}\times \mathbb{R}^4\rightarrow \mathbb{C},
\end{equation}
where $i$ denotes the imaginary unit and $f$ satisfies:

$(H_1)$ $ f(t)\in \mathbb{R}$ for $t\in \mathbb{R}$ and $f(e^{i\theta }z)=e^{i\theta }f(z)$ for any $\theta \in \mathbb{R}$, $z\in \mathbb{C}$;

$(H_2)$ $\displaystyle F(z)=\int_{0}^{|z|}f(t)dt$ for any $z\in \mathbb{C}$.\\
To look for standing wave solutions $\psi(t,x)=e^{-i\lambda t}u(x)$ for \eqref{back}, where $\lambda\in \mathbb{R}$ and $u\in H^2(\mathbb{R}^4)$ is a time-independent real valued function, one leads to problem \eqref{abs1}. In particular, we are interested in looking for ground state solutions,  i.e., solutions minimizing the associated energy functional among all nontrivial solutions.
If $\lambda\in \mathbb{R}$ is a fixed parameter, there are many results for equation \eqref{abs1} by variational methods, we refer to \cite{BS3,CGT,GS,MS1,MS2,MS3,RS1} and references therein.

Another alternative choice is to find solutions of \eqref{abs1} with prescribed mass, and $\lambda\in \mathbb{R}$ arises as a Lagrange multiplier. This type of solution is called normalized solution, and this approach is particularly meaningful from the physical point of view, since, in addition to there is a conservation of mass by $(H_1)$ and $(H_2)$, the mass has often an important physical meaning, e.g. it shows the power supply in nonlinear optics, or the total number of atoms in Bose-Einstein condensation.

For the nonlinear Schr\"{o}dinger equation with normalization constraint
\begin{align}\label{e1.3}
  \begin{split}
  \left\{
  \begin{array}{ll}
   -\Delta u=\lambda u+f(u),
    \quad\mbox{in}\ \ \mathbb{R}^N,  \\
    \displaystyle\int_{\mathbb{R}^N}|u|^2dx=c^2.
    \end{array}
    \right.
  \end{split}
  \end{align}
If $f(u)=|u|^{p-2}u$, by the Gagliardo-Nirenberg inequality \cite{Nirenberg1}, $\bar{p}:=2+\frac{4}{N}$ is called the $L^2$-critical exponent. In this case,
the associated energy functional of \eqref{e1.3} is defined by
$$
\mathcal{J}_1(u)=\frac{1}{2}\int_{\mathbb{R}^N}|\nabla u|^2dx-\frac{1}{p}\int_{\mathbb{R}^N}|u|^pdx.
$$
If the problem is $L^2$-subcritical, i.e., $p\in(2,2+\frac{4}{N})$, then $\mathcal{J}_1$ is bounded from below on
  $\widehat{S}(c):=\{u\in H^1(\mathbb{R}^N):\int_{\mathbb{R}^N}|u|^2dx=c^2\}$, so the ground state solution of \eqref{e1.3} can be found as a global minimizer of $\mathcal{J}_1$ on $\widehat{S}(c)$. Stuart \cite{S1,S2} first obtained the existence of normalized solutions for \eqref{e1.3} by the bifurcation theory. Other results for the $L^2$-subcritical problems can be found, we refer readers to \cite{CL,Lions2,Shibata} and references therein.

However, if the problem is $L^2$-critical or supercritical, i.e., $p\in[2+\frac{4}{N},2^*)$, where $2^*=\infty$ if $N\leq2$ and $2^*=\frac{2N}{N-2}$ if $N\geq3$, then $\mathcal{J}_1$ is unbounded from below on $\widehat{S}(c)$, so it seems impossible to search for a global minimizer to obtain a solution of \eqref{e1.3}.
Furthermore,  since $\lambda\in \mathbb{R}$ is unknown and $H^1(\mathbb{R}^N)\hookrightarrow L^2(\mathbb{R}^N)$ is not compact, some classical methods cannot be directly used to prove the boundedness and compactness of any $(PS)$ sequence.
This case was first studied by Jeanjean in \cite{Jeanjean}. For quite a long time, the work of Jeanjean \cite{Jeanjean} is the only one in this aspect, and the idea of \cite{Jeanjean} has been exploited and further developed in the search for normalized solutions when the energy functional is unbounded from below on the $L^2$ constraint, see \cite{AJM,BS0,BS1,BS2,BM,DZ,JL,Li1,Soave1,Soave2,WW, ZZZ} and so on for normalized solutions in $\mathbb{R}^N$, \cite{NTV1,PV} for normalized solutions in bounded domains, and \cite{BLL,Li2,LY1,YCRS,YCT} for normalized solutions of Choquard equations. It is worth pointing out that in \cite{BS1,BS2},
Bartsch and Soave presented a
new approach that is based on a natural
constraint associated to the problem and proved the existence of normalized solutions for \eqref{e1.3} by using a minimax principle based on the homotopy stable family.
Borrowing some arguments from Bartsch and Soave \cite{BS1,BS2}, Jeanjean and Lu \cite{JL} also complemented and generalized the results of  \cite{Jeanjean}.
Comparing the above two methods, we find that Jeanjean's method   requires the energy functional has some scaling-type property.

Let us mention the following mixed dispersion biharmonic nonlinear Schr\"{o}dinger equation with prescribed $L^2$-norm constraint
\begin{align}\label{biharmonic}
  \begin{split}
  \left\{
  \begin{array}{ll}
   \Delta^2u-\beta\Delta u=\lambda u+f(u),
    \quad\mbox{in}\ \ \mathbb{R}^N,  \\
    \displaystyle\int_{\mathbb{R}^N}|u|^2dx=c^2,\quad u\in H^2(\mathbb{R}^N).
    \end{array}
    \right.
  \end{split}
  \end{align}
In this case, this kind of problem gives a new $L^2$-critical exponent $\bar{q}:=2+\frac{8}{N}$.
If $\beta>0$, $f(u)=|u|^{q-2}u$, and $q\in (2,2+\frac{8}{N})$, Bonheure et al. \cite{BCdN} studied \eqref{biharmonic} and established the existence, qualitative properties of minimizers. By using the minimax principle, Bonheure et al. \cite{BCGJ} obtained the existence of ground state solutions, radial positive solutions, and the multiplicity of radial solutions for \eqref{biharmonic} when $q\in (2+\frac{8}{N},4^*)$, where $4^*=\infty$ if $N\leq4$ and $4^*=\frac{2N}{N-4}$ if $N\geq5$.
If $\beta<0$, the problem is more involved, see \cite{LZZ,BFJ} for $q\in (2,2+\frac{8}{N})$ and \cite{LY2} for $q\in (2+\frac{8}{N},4^*)$. Moreover,
Luo and Zhang \cite{LZ1} studied normalized solutions of \eqref{biharmonic} with a general nonlinear term $f$,
where $\beta\in \mathbb{R}$ and $f$ satisfies the suitable $L^2$-subcritical assumptions.

It is well known that the embedding $H^2(\mathbb{R}^4)\hookrightarrow L^\infty(\mathbb{R}^4)$ is not compact, by using Jeanjean's trick \cite{Jeanjean}, the authors have studied the existence of normalized solutions for \eqref{biharmonic} with a Choquard nonlinear term involving exponential critical growth in $\mathbb{R}^4$, see \cite{CW} for more detail.
Inspired by \cite{BS1,BS2,JL},  we are going to give an another view to study the normalized solution of equation \eqref{abs1}-\eqref{abs2}. Compared with the previous work \cite{CW}, there are some differences and features of this paper as follows:

$(i)$ Different from \cite{CW}, by using the minimax principle based on the homotopy stable family,
we work directly in $H^2(\mathbb{R}^4)$ to obtain a $(PS)_{E(c)}$ sequence, where $E(c)$ is the ground state energy defined in (\ref{E(c)}). Then by showing that $E(c)$ is nonincreasing with respect to $c$ and strictly decreasing for some $c$, we overcome the compactness.

$(ii)$ Based on the Adams functions \cite{LY3}, we give a more natural condition, see $(f_5)$, to estimate the upper bound of $E(c)$, which is crucial to rule out the trivial case for the weak limit of a $(PS)_{E(c)}$ sequence. However, this Adams functions cannot be directly applied to study normalized solutions, because the functions are abstract in some ranges and the $L^2$-norms of $u$, $\nabla u$, $\Delta u$ are given as $O(1/\log n)$, after a
normalization, it is not conducive for us to obtain a refined estimation.
With regard to this point, we make some appropriate modifications to the Adams functions, then using the following condition $(f_5)$, through a careful calculation, we complete the estimation.

Now, we introduce the precise assumptions on $f$.
Assume
that $f$ satisfies:

$(f_1)$ $f\in C(\mathbb{R},\mathbb{R})$ is odd or even, $\displaystyle\lim\limits_{t\rightarrow0} \frac{|f(t)|}{|t|^\nu}=0$ for some $\nu>2-\frac{\mu}{4}$;

$(f_2)$ $f$ has exponential critical growth at infinity, i.e.,
\begin{align*}
  \begin{split}
  \lim\limits_{|t|\rightarrow+\infty}\frac{|f(t)|}{e^{\alpha t^2}}=\left\{
  \begin{array}{ll}
  0,&\quad \text {for} \,\,\,\alpha>32\pi^2,\\
  +\infty,&\quad \text {for} \,\,\,0<\alpha<32\pi^2;
    \end{array}
    \right.
  \end{split}
  \end{align*}

$(f_3)$ There exists a constant $\theta>3-\frac{\mu}{4}$ such that
\begin{equation*}
  0<\theta F(t)\leq tf(t), \,\,\,\text{for all $t\in \mathbb{R}\backslash \{0\}$};
\end{equation*}

$(f_4)$ There exist $M_0>0$ and $R_0>0$ such that $F(t)\leq M_0|f(t)|$, \text{for all $|t|\geq R_0$}.

$(f_5)$ There exists $\varrho>0$ such that $\liminf\limits_{t\rightarrow+\infty}\frac{f(t)}{e^{32\pi^2 t^2}}\geq\varrho$. 

$(f_6)$ For any $t\in \mathbb{R}\backslash \{0\}$, let $\overline{F}(t):=f(t)t-(2-\frac{\mu}{4})F(t)$,  $f'(t)$
exists, and
\begin{equation*}
(3-\frac{\mu}{4})F(s)\overline{F}(t)<F(s)\overline{F}'(t)t+\overline{F}(s)(\overline{F}(t)-F(t)),\quad \text{for any $s,t\in \mathbb{R}\backslash \{0\}$};
\end{equation*}

$(f_7)$ $\frac{\overline{F}(t)}{|t|^{3-\frac{\mu}{4}}}$ is non-increasing in $(-\infty,0)$ and non-decreasing in $(0,+\infty)$.

Our main result can be stated as follows:
\begin{theorem}\label{th2}
Assume that $f$ satisfies $(f_1)-(f_5)$, if $(f_6)$ or $(f_7)$ holds, then there exists $\beta_*>0$ such that for any $\beta\in [0,\beta_*)$, \eqref{abs1}-\eqref{abs2} admit a positive ground state solution.
\end{theorem}

\begin{remark}
{\rm
A typical example satisfying $(f_1)-(f_7)$ is $
  f(t)=|t|^{p-2}te^{32\pi t^2}$
for any $p>\max\{3,\nu+1,\theta\}$,
 and if $f'$ exists, then  $(f_7)$ implies $(f_6)$.
}
\end{remark}

We define the energy functional $\mathcal J:H^2(\mathbb{R}^4)\rightarrow\mathbb{R}$ by
\begin{align*}
\mathcal J(u)=\frac{1}{2}\int_{\mathbb{R}^4}|\Delta u|^2dx+\frac{\beta}{2}\int_{\mathbb{R}^4}|\nabla u|^2dx-\frac{1}{2}\int_{\mathbb{R}^4}(I_\mu*F(u))F(u)dx,
\end{align*}
where $ H^2(\mathbb{R}^4)=\Big\{u\in L^2(\mathbb{R}^4):\nabla u\in L^2(\mathbb{R}^4),\ \Delta u\in L^2(\mathbb{R}^4)\Big\}$
endowed with the norm
\begin{equation*}
  \|u\|=(\|\Delta u\|_2^2+2\|\nabla u\|_2^2+\|u\|_2^2)^{\frac{1}{2}}.
\end{equation*}
Recalling the following interpolation inequality
\begin{equation*}
  \int_{\mathbb{R}^4}|\nabla u|^2dx\leq  \Big(\int_{\mathbb{R}^4}|\Delta u|^2dx\Big)^{\frac{1}{2}}\Big(\int_{\mathbb{R}^4}| u|^2dx\Big)^{\frac{1}{2}},
\end{equation*}
we can see that $\|\cdot\|$ is equivalent to the norm $ \|u\|_{H^2(\mathbb{R}^4)}:=(\|\Delta u\|_2^2+\|u\|_2^2)^{\frac{1}{2}}$. By $(f_1)$ and $(f_2)$, fix $q>2$, for any $\xi>0$ and $\alpha>32\pi^2$, there exists a constant $C_\xi>0$ such that
\begin{equation*}\label{fcondition1}
  |f(t)|\leq \xi|t|^\nu+C_\xi|t|^q(e^{\alpha t^2}-1),\quad\text{for all $t\in \mathbb{R}$},
\end{equation*}
and using $(f_3)$, we have
\begin{equation}\label{fcondition2}
|F(t)|\leq \xi|t|^{\nu+1}+C_\xi|t|^{q+1}(e^{\alpha t^2}-1),\quad\text{for all $t\in \mathbb{R}$}.
\end{equation}
By \eqref{fcondition2}, using the Hardy-Littlewood-Sobolev inequality \cite{LL} and the Adams inequality \cite{Y}, we obtain
$\mathcal{J}$ is well defined in $H^2(\mathbb{R}^4)$ and $\mathcal{J}\in C^1(H^2(\mathbb{R}^4),\mathbb{R})$ with
\begin{align*}
\langle\mathcal J'(u),v\rangle=&\int_{\mathbb{R}^4}\Delta u \Delta vdx+\beta\int_{\mathbb{R}^4}\nabla u \cdot \nabla vdx-\int_{\mathbb{R}^4}(I_\mu*F(u))f(u)vdx,
\end{align*}
for any $u, v\in H^2(\mathbb{R}^4)$.
For any $c>0$, set
\begin{equation*}
 S(c):=\Big\{u\in H^2(\mathbb{R}^4):\int_{\mathbb{R}^4}|u|^2dx=c^2\Big\}.
\end{equation*}
Define
\begin{equation}\label{E(c)}
  E(c):=\inf\limits_{u\in \mathcal{P}(c)}\mathcal J(u),
\end{equation}
where $\mathcal{P}(c)$ is the Pohozaev manifold defined by
$$\mathcal{P}(c)=\Big\{u\in S(c):P(u)=0\Big\}
$$
with
\begin{align*}
  P(u)&=2\int_{\mathbb{R}^4}|\Delta u |^2\ dx+\beta\int_{\mathbb{R}^4}|\nabla u|^2 dx+\frac{8-\mu}{2} \int_{\mathbb{R}^4} (I_\mu*F(u))F(u)dx-2\int_{\mathbb{R}^4} (I_\mu*F(u))f(u)u dx.
\end{align*}
By Lemma \ref{equi}, we know $\mathcal{P}(c)$ is nonempty, and from Lemma \ref{Pohozaev}, we can see that any critical point of $\mathcal{J}$ on $S(c)$ stays in $\mathcal{P}(c)$, thus any critical point $u$ of $\mathcal{J}$ on $S(c)$ with $\mathcal{J}(u)=E(c)$ is a ground state solution of \eqref{abs1}-\eqref{abs2}.

For any $s\in \mathbb{R}$ and $u\in H^2(\mathbb{R}^4)$, we define
\begin{equation*}
  \mathcal{H}(u,s)(x):=e^{2s}u(e^sx),\quad\text{for a.e. $x\in \mathbb{R}^4$}.
\end{equation*}
For simplicity, we always write $\mathcal{H}(u,s)$. One can easily check that $\| \mathcal{H}(u,s)\|_2=\|u\|_2$ for any $s\in \mathbb{R}$, and $\mathcal{H}(u,s_1+s_2)=\mathcal{H}(\mathcal{H}(u,s_1),s_2)=\mathcal{H}(\mathcal{H}(u,s_2),s_1)$ for any $s_1,s_2\in \mathbb{R}$.

\begin{remark}
{\rm When $f$ satisfies $(f_1)-(f_3)$, if $(f_6)$ or $(f_7)$ holds, from Lemma \ref{inf}, we can see that $E(c)$ is well defined and strictly positive. As we will see in Lemma \ref{pssequencehomo}, using $(f_3)$, $(f_4)$ and the minimax principle based on the homotopy stable family, we can construct a bounded $(PS)_{E(c)}$ sequence $\{u_n\}$, up to a subsequence and up to translations in $\mathbb{R}^4$, its weak limit $u_c\in H^2(\mathbb{R}^4)$ is a nontrivial weak solution to \eqref{abs1}.
Moreover, $c_1:=\|u_c\|_2\in (0,c]$ and $P(u_c)=0$. Using $(f_3)$ and Fatou lemma, we prove that $E(c_1)\leq\mathcal{J}(u)\leq E(c)$. Hence, it is clear that if we can show that $E(c)$ is strictly decreasing for some $c$, then we have $\mathcal{J}(u)=E(c)$ and $\|u\|_2=c$, i.e., $u$ is a ground state solution of \eqref{abs1}-\eqref{abs2}. Therefore, from this observation, it is natural to study the monotonicity of the function $c\mapsto E(c)$.}
\end{remark}

In this paper, we use
$C$ to denote any positive constants possibly different from line to line.  $B_r(x)$ denotes the open ball centered at $x\in \mathbb{R}^4$ with radius $r>0$.
$(E^*, \|\cdot\|_{*})$ is the dual space of Banach space $(E, \|\cdot\|)$. $L^p(\mathbb{R}^4)$ is the usual Lebesgue space endowed with the norm $\|u\|_{p}=(\int_{\mathbb{R}^4}|u|^pdx)^{\frac{1}{p}}$ when $1<p<\infty$,  $\|u\|_\infty=\inf\{C>0, |u(x)|\leq C \,\,\text{a.e. in} \,\,\mathbb{R}^4\}$.

Our paper is arranged as follows. The forthcoming section contains some preliminary results. In Section \ref{estimation}, we introduce the modified Adams functions, and estimate the upper bound of $E(c)$.
The monotonicity of the function $c\mapsto E(c)$ is studied in Section \ref{Beha}.  In Section \ref{mini}, we use the minimax principle based on the homotopy stable family to construct a bounded $(PS)_{E(c)}$ sequence. Section \ref{main} is devoted to the proof of Theorem \ref{th2}.

\section{Preliminaries}\label{sec preliminaries}

In this section, we give some preliminaries.
For the nonlocal type problems with Riesz potential, an important inequality due to the Hardy-Littlewood-Sobolev inequality will be used in the following.
\begin{proposition}\cite[Theorem 4.3]{LL}\label{HLS}
Assume that $1<r$, $t<\infty$, $0<\mu<4$ and
$
\frac{1}{r}+\frac{\mu}{4}+\frac{1}{t}=2.
$
Then there exists $C(\mu,r,t)>0$ such that
\begin{align}\label{HLSin}
\Big|\int_{\mathbb R^{4}}(I_{\mu}\ast g(x))h(x)dx\Big|\leq
C(\mu,r,t)\|g\|_r
\|h\|_t
\end{align}
for all $g\in L^r(\mathbb{R}^4)$ and $h\in L^t(\mathbb{R}^4)$.
\end{proposition}

\begin{lemma}(Cauchy-Schwarz type inequality) \cite{Matt}
For $g,h\in L_{loc}^1(\mathbb{R}^4)$, there holds
\
\begin{equation}\label{CS}
  \int_{\mathbb R^{4}}(I_{\mu}\ast |g(x)|)|h(x)|dx\leq \Big(\int_{\mathbb R^{4}}(I_{\mu}\ast |g(x)|)|g(x)|dx\Big)^{\frac{1}{2}}\Big(\int_{\mathbb R^{4}}(I_{\mu}\ast |h(x)|)|h(x)|dx\Big)^{\frac{1}{2}}.
\end{equation}
\end{lemma}

\begin{lemma}(Gagliardo-Nirenberg inequality) \cite{Nirenberg1}
For any $u\in H^2(\mathbb{R}^4)$ and $p\geq2$, then it holds
\begin{equation}\label{GNinequality}
  \|u\|_p\leq B_{p}\|\Delta u\|_2^{\frac{p-2}{p}}\| u\|_2^{\frac{2}{p}},
\end{equation}
where $B_{p}$ is a constant depending on $p$.
\end{lemma}

\begin{lemma}\label{adams}
(i) \cite[Theorem 1.2]{Y} If $\alpha>0$ and $u\in H^2(\mathbb{R}^4)$, then
\begin{equation*}
  \int_{\mathbb{R}^4}(e^{\alpha u^2}-1)dx<+\infty;
\end{equation*}
(ii) \cite[Proposition 7]{MSV} There exists a constant $C>0$ such that
\begin{equation*}
  \sup\limits_{u\in H^2(\mathbb{R}^4), \|\Delta u\|_2\leq1}\int_{\mathbb{R}^4}(e^{\alpha u^2}-1)dx\leq C
\end{equation*}
for all $0<\alpha\leq 32\pi^2$.
\end{lemma}
\begin{lemma}\cite[Lemma 4.8]{Kav}\label{weakcon}
Let $\Omega\subseteq\mathbb R^{4}$ be any open set. For $1<s<\infty$, let $\{u_n\}$ be bounded in $L^s(\Omega)$ and $u_n(x)\rightarrow u(x)$ a.e. in $\Omega$. Then $u_n(x)\rightharpoonup u(x) \,\,in \,\,L^s(\Omega)$.
\end{lemma}
\begin{definition}\cite[Definition 3.1]{G}
Let $B$ be a closed subset of $X$. A class of $\mathcal{F}$ of compact subsets of $X$ is a homotopy stable family with boundary $B$ provided

$(i)$ Every set in $\mathcal{F}$ contains $B$;

$(ii)$ For any set $A\in \mathcal{F}$ and any $\eta\in C([0,1]\times X,X)$ satisfying $\eta(t,v)=v$ for all $(t,v)\in (\{0\}\times X)\cup([0,1]\times B)$, one has $\eta(\{1\}\times A)\in\mathcal{F}$.
\end{definition}
\begin{lemma}\cite[Theorem 3.2]{G}\label{Ghouss}
Let $\psi$ be a $C^1$-functional on a complete connected $C^{1}$-Finsler manifold $X$ (without boundary), and consider a homotopy stable family $\mathcal{F}$ with a closed boundary $B$. Set $\widetilde{c} =\widetilde{c}(\psi,\mathcal{F})=\inf\limits_{A\in \mathcal{F}}\max\limits_{v\in A}\psi(v)$ and suppose that
\begin{equation*}
  \sup\psi(B)<\widetilde{c}.
\end{equation*}
Then, for any sequence of sets $\{A_n\} \subset\mathcal{F}$ satisfying $\lim\limits_{n\rightarrow\infty}\sup\limits_{v\in A_n}\psi(v)=\widetilde{c}$, there exists a sequence $\{v_n\}\subset X$ such that

$(i)$ $\lim\limits_{n\rightarrow\infty}\psi(v_n)=\widetilde{c}$;

$(ii)$ $\lim\limits_{n\rightarrow\infty}\|\psi'(v_n)\|_*=0$;

$(iii)$ $\lim\limits_{n\rightarrow\infty}dist(v_n,A_n)=0$.
\end{lemma}

We observe that $B=\emptyset$ is admissible, it is sufficient to follow the usual convention of defining $\sup\psi(\emptyset)=-\infty$. Since $G(u):=\|u\|_2^2-c^2$ is of class $C^1$ and for any $u\in S(c)$, we have $\langle G'(u),u\rangle=2c^2>0$. Therefore, by the implicit function theorem, $S(c)$ is a $C^{1}$-Finsler manifold.
\begin{lemma}\label{biancon}
Assume that $u_n\rightarrow u$ in $H^2(\mathbb{R}^4)$, and $s_n\rightarrow s$ in $\mathbb{R}$ as $n\rightarrow\infty$, then $\mathcal{H}(u_n,s_n)\rightarrow \mathcal{H}(u,s)$ in $H^2(\mathbb{R}^4)$ as $n\rightarrow\infty$.
\end{lemma}
\begin{proof}
We first prove that $\mathcal{H}(u_n,s_n)\rightharpoonup\mathcal{H}(u,s)$ in $H^2(\mathbb{R}^4)$. Taking any $\phi\in C_0^\infty(\mathbb{R}^4)$, and for a compact set $K$ containing the support of $\phi(e^{-s_n}\cdot)$ for any $n$ large enough, using the Lebesgue dominated convergence theorem, we get
\begin{align*}
  \int_{\mathbb{R}^4}e^{2s_n}u_n(e^{s_n}x)\phi(x)dx&= \int_{K}e^{-2s_n}u_n(y)\phi(e^{-s_n}y)dy\\
  &\rightarrow \int_{K}e^{-2s}u(y)\phi(e^{-s}y)dy=\int_{\mathbb{R}^4}e^{2s}u(e^{s}x)\phi(x)dx,\quad\text{as $n\rightarrow\infty$}.
\end{align*}
Similarly, we can show that for any $i=1,2,3,4$ and $\phi\in C_0^\infty(\mathbb{R}^4)$
\begin{align*}
  \int_{\mathbb{R}^4}\phi \frac{\partial (\mathcal{H}(u_n,s_n))}{\partial {x_i}}dx\rightarrow \int_{\mathbb{R}^4}\phi \frac{\partial (\mathcal{H}(u,s))}{\partial {x_i}}dx,\quad\text{as $n\rightarrow\infty$},
\end{align*}
and
\begin{align*}
  \int_{\mathbb{R}^4}\phi \frac{\partial ^2 (\mathcal{H}(u_n,s_n))}{\partial {x_i}^2}dx\rightarrow \int_{\mathbb{R}^4}\phi \frac{\partial ^2(\mathcal{H}(u,s))}{\partial {x_i}^2}dx,\quad\text{as $n\rightarrow\infty$}.
\end{align*}
As a consequence, $\mathcal{H}(u_n,s_n)\rightharpoonup\mathcal{H}(u,s)$ in $H^2(\mathbb{R}^4)$. Furthermore,
\begin{align*}
  \|\mathcal{H}(u_n,s_n)\|^2&= \int_{\mathbb{R}^4}|u_n|^2dx+2e^{2s_n}\int_{\mathbb{R}^4}|\nabla u_n|^2dx+e^{4s_n}\int_{\mathbb{R}^4}|\Delta u_n|^2dx\\
  &\rightarrow \int_{\mathbb{R}^4}|u|^2dx+2e^{2s}\int_{\mathbb{R}^4}|\nabla u|^2dx+e^{4s}\int_{\mathbb{R}^4}|\Delta u|^2dx=\|\mathcal{H}(u,s)\|^2,\quad\text{as $n\rightarrow\infty$},
\end{align*}
thus we can get $\|\mathcal{H}(u_n,s_n)-\mathcal{H}(u,s)\|\rightarrow0$ in $H^2(\mathbb{R}^4)$ as $n\rightarrow\infty$.
\end{proof}

\begin{lemma}\label{strong}
Assume that $(f_1)-(f_4)$ hold, let $\{u_n\}\subset S(c)$ be a bounded sequence in $H^2(\mathbb{R}^4)$,
if $u_n\rightharpoonup u$ in $H^2(\mathbb{R}^4)$ and
\begin{equation}\label{condition}
  \int_{\mathbb{R}^4} (I_\mu*F(u_n))f(u_n)u_n dx\leq K_0
\end{equation}
for some constant $K_0>0$, then for any $\phi\in C_0^\infty(\mathbb{R}^4)$, we have 
\begin{equation}\label{strong1}
  \int_{\mathbb{R}^4}\Delta u_n\Delta\phi dx \rightarrow  \int_{\mathbb{R}^4}\Delta u\Delta\phi dx,\,\,\,\text{as}\,\,\, n\rightarrow\infty,
\end{equation}
\begin{equation}\label{strong2}
  \int_{\mathbb{R}^4}\nabla u_n\cdot\nabla\phi dx \rightarrow  \int_{\mathbb{R}^4}\nabla u\cdot\nabla\phi dx,\,\,\,\text{as}\,\,\, n\rightarrow\infty,
\end{equation}
\begin{equation}\label{strong3}
  \int_{\mathbb{R}^4}u_n\phi dx \rightarrow  \int_{\mathbb{R}^4}u\phi dx,\,\,\,\text{as}\,\,\, n\rightarrow\infty,
\end{equation}
and
\begin{equation}\label{strong4}
  \int_{\mathbb{R}^4} (I_\mu*F(u_n))f(u_n)\phi dx\rightarrow \int_{\mathbb{R}^4} (I_\mu*F(u))f(u)\phi dx,\,\,\,\text{as}\,\,\, n\rightarrow\infty.
\end{equation}
\end{lemma}

\begin{proof}
For any fixed $v\in H^2(\mathbb{R}^4)$, define
\begin{equation*}
  f_v(u):=\int_{\mathbb{R}^4}\Delta u\Delta v dx,\quad g_v(u):=\int_{\mathbb{R}^4}\nabla u\nabla v dx,\quad h_v(u):=\int_{\mathbb{R}^4}u v dx\quad \text{for every $u \in H^2(\mathbb{R}^4)$}.
\end{equation*}
Then, by the H\"{o}lder inequality, we have
\begin{equation*}
  |f_v(u)|,|g_v(u)|,|h_v(u)|\leq\|v\|\|u\|,
\end{equation*}
this yields that $f_v$, $g_v$, and $h_v$ are continuous linear functionals on $H^2(\mathbb{R}^4)$. Thus, by $u_n\rightharpoonup u$ in $H^2(\mathbb{R}^4)$ and $C_0^\infty(\mathbb{R}^4)$ is dense in $H^2(\mathbb{R}^4)$, we prove that \eqref{strong1}-\eqref{strong3} hold. Now, adopting some ideas of Lemma $4.8$ in \cite{QT}, we will prove \eqref{strong4}.
By Fatou lemma and \eqref{condition}, we have
\begin{equation}\label{condition1}
  \int_{\mathbb{R}^4} (I_\mu*F(u))f(u)u dx\leq K_0.
\end{equation}
Let $\Omega=supp \phi$. For any given $\varepsilon>0$, let $M_\varepsilon=\frac{K_0\|\phi\|_\infty}{\varepsilon}$, then it follows from $(f_3)$, \eqref{condition} and \eqref{condition1} that,
\begin{equation}\label{begin1}
  \int\limits_{\{|u_n|\geq M_\varepsilon\}\cup \{|u|=M_\varepsilon\}} (I_\mu*F(u_n))|f(u_n)\phi| dx\leq \frac{2\varepsilon}{K_0}\int\limits_{|u_n|\geq \frac{M_\varepsilon}{2}} (I_\mu*F(u_n))f(u_n)u_n dx\leq 2\varepsilon,
\end{equation}
and
\begin{equation}\label{begin2}
  \int\limits_{|u|\geq M_\varepsilon} (I_\mu*F(u))|f(u)\phi| dx\leq \frac{\varepsilon}{K_0}\int\limits_{|u|\geq M_\varepsilon} (I_\mu*F(u))f(u)u dx\leq \varepsilon.
\end{equation}
Since $|f(u_n)|{\chi_{|u_n|\leq M_\varepsilon}}\rightarrow |f(u)|{\chi_{|u|\leq M_\varepsilon}}$ a.e. in $\Omega\backslash D_\varepsilon$, where $D_\varepsilon=\{x\in \Omega:|u(x)|=M_\varepsilon\}$, and for any $x\in \Omega$,
\begin{equation*}
  |f(u_n)|{\chi_{|u_n|\leq M_\varepsilon}}\leq \max\limits_{|t|\leq M_\varepsilon}|f(t)|<\infty,
\end{equation*}
using the Lebesgue dominated convergence theorem, we get
\begin{equation}\label{lebe}
  \lim\limits_{n\rightarrow\infty}\int\limits_{\{\Omega\backslash D_\varepsilon\}\cap\{|u_n|\leq M_\varepsilon\}}|f(u_n)|^{\frac{8}{8-\mu}}dx=\int\limits_{\{\Omega\backslash D_\varepsilon\}\cap\{|u|\leq M_\varepsilon\}}|f(u)|^{\frac{8}{8-\mu}}dx.
\end{equation}
Choosing $K_\varepsilon>R_0$ such that
\begin{equation}\label{give1}
 \|\phi\|_\infty\Big(\frac{M_0K_0}{K_\varepsilon}\Big)^{\frac{1}{2}}\Big(\int\limits_{\Omega}|f(u)|^{\frac{8}{8-\mu}}dx\Big)^{\frac{8-\mu}{8}}<\varepsilon
\end{equation}
and
\begin{equation}\label{give2}
  \int\limits_{|u|\leq M_\varepsilon} (I_\mu*(F(u)\chi _{|u|\geq K_\varepsilon}))|f(u)\phi| dx<\varepsilon.
\end{equation}
Then from $(f_4)$, \eqref{HLS}, \eqref{CS}, \eqref{lebe} and \eqref{give1}, one has
\begin{align}\label{give3}
  &\int\limits_{\{|u_n|\leq M_\varepsilon\}\cap \{|u|\neq M_\varepsilon\}} (I_\mu*(F(u_n)\chi_{|u_n|\geq K_\varepsilon}))|f(u_n)\phi| dx\nonumber\\
  &\leq \|\phi\|_\infty \int\limits_{\Omega\backslash D_\varepsilon} (I_\mu*(F(u_n)\chi_{|u_n|\geq K_\varepsilon}))|f(u_n)|\chi_{|u_n|\leq M_\varepsilon} dx\nonumber\\
  & \leq\|\phi\|_\infty\Big(\int\limits_{\mathbb{R}^4} (I_\mu*(F(u_n)\chi_{|u_n|\geq K_\varepsilon}))F(u_n)\chi_{|u_n|\geq K_\varepsilon} dx\Big)^{\frac{1}{2}}\nonumber\\
  &\quad\times \Big(\int\limits_{\mathbb{R}^4} (I_\mu*(|f(u_n)|\chi_{\{\Omega\backslash D_\varepsilon\}\cap\{|u_n|\leq M_\varepsilon\}} ))|f(u_n)|\chi_{\{\Omega\backslash D_\varepsilon\}\cap\{|u_n|\leq M_\varepsilon\}} dx\Big)^{\frac{1}{2}}\nonumber\\
  &\leq \|\phi\|_\infty\Big(\int\limits_{|u_n|\geq K_\varepsilon} (I_\mu*F(u_n))F(u_n) dx\Big)^{\frac{1}{2}}\nonumber\\
  &\quad\times \Big(\int\limits_{\mathbb{R}^4} (I_\mu*(|f(u_n)|\chi_{\{\Omega\backslash D_\varepsilon\}\cap\{|u_n|\leq M_\varepsilon\}} ))|f(u_n)|\chi_{\{\Omega\backslash D_\varepsilon\}\cap\{|u_n|\leq M_\varepsilon\}} dx\Big)^{\frac{1}{2}}\nonumber\\
  &\leq C\|\phi\|_\infty\Big(\int\limits_{|u_n|\geq K_\varepsilon} (I_\mu*F(u_n))F(u_n) dx\Big)^{\frac{1}{2}}\Big(\int\limits_{\{\Omega\backslash D_\varepsilon\}\cap\{|u_n|\leq M_\varepsilon\}}|f(u_n)|^{\frac{8}{8-\mu}}dx\Big)^{\frac{8-\mu}{8}}\nonumber\\
  &\leq C\|\phi\|_\infty\Big(\frac{M_0}{K_\varepsilon}\int\limits_{|u_n|\geq K_\varepsilon} (I_\mu*F(u_n))f(u_n)u_n dx\Big)^{\frac{1}{2}}\Big(\int\limits_{\Omega}|f(u)|^{\frac{8}{8-\mu}}dx\Big)^{\frac{8-\mu}{8}}\nonumber\\
  & \leq C\|\phi\|_\infty\Big(\frac{M_0K_0}{K_\varepsilon}\Big)^{\frac{1}{2}}\Big(\int\limits_{\Omega}|f(u)|^{\frac{8}{8-\mu}}dx\Big)^{\frac{8-\mu}{8}}<\varepsilon.
\end{align}
For any $x\in \mathbb{R}^4$, set
\begin{equation*}
  \varpi_n(x):=\int_{\mathbb{R}^4}\frac{|F(u_n)|\chi_{|u_n|\leq K_\varepsilon}}{|x-y|^\mu}dy\quad \text{and}\quad \varpi(x):=\int_{\mathbb{R}^4}\frac{|F(u)|\chi_{|u|\leq K_\varepsilon}}{|x-y|^\mu}dy.
\end{equation*}
Then from \eqref{fcondition2}, for any $x\in \mathbb{R}^4$ and $R>0$, one has
\begin{align*}
  |\varpi_n(x)-\varpi(x)|&\leq \int_{\mathbb{R}^4}\frac{\Big||F(u_n)|\chi_{|u_n|\leq K_\varepsilon}-|F(u)|\chi_{|u|\leq K_\varepsilon}\Big|}{|x-y|^\mu}dy\\
  &\leq \Big( \int_{|x-y|\leq R}{\Big||F(u_n)|\chi_{|u_n|\leq K_\varepsilon}-|F(u)|\chi_{|u|\leq K_\varepsilon}\Big|^{\frac{4+\mu}{4-\mu}}}dy\Big)^{\frac{4-\mu}{4+\mu}}\Big(\int_{|x-y| \leq R}|x-y|^{-\frac{4+\mu}{2}}dy\Big)^{\frac{2\mu}{4+\mu}}\\
  &\quad+\Big( \int_{|x-y|\geq R}{\Big||F(u_n)|\chi_{|u_n|\leq K_\varepsilon}-|F(u)|\chi_{|u|\leq K_\varepsilon}\Big|^{\frac{8-\mu}{8-2\mu}}}dy\Big)^{\frac{8-2\mu}{8-\mu}}\Big(\int_{|x-y| \geq R}|x-y|^{\mu-8}dy\Big)^{\frac{\mu}{8-\mu}}\\
  &\leq \Big(\frac{4\pi^2}{4-\mu}R^{\frac{4-\mu}{2}}\Big)^{\frac{2\mu}{4+\mu}}\Big( \int_{|x-y|\leq R}{\Big||F(u_n)|\chi_{|u_n|\leq K_\varepsilon}-|F(u)|\chi_{|u|\leq K_\varepsilon}\Big|^{\frac{4+\mu}{4-\mu}}}dy\Big)^{\frac{4-\mu}{4+\mu}}\\&\quad+
  \Big(\frac{2\pi^2}{(4-\mu)R^{4-\mu}}\Big)^{\frac{\mu}{8-\mu}}\Big( \int_{|x-y|\geq R}{\Big||F(u_n)|\chi_{|u_n|\leq K_\varepsilon}-|F(u)|\chi_{|u|\leq K_\varepsilon}\Big|^{\frac{8-\mu}{8-2\mu}}}dy\Big)^{\frac{8-2\mu}{8-\mu}}.
\end{align*}
Similar to \eqref{lebe}, we can prove that
\begin{equation*}
   \int_{|x-y|\leq R}{\Big||F(u_n)|\chi_{|u_n|\leq K_\varepsilon}-|F(u)|\chi_{|u|\leq K_\varepsilon}\Big|^{\frac{4+\mu}{4-\mu}}}dy\rightarrow0,\quad\text{as $n\rightarrow\infty$}.
\end{equation*}
By \eqref{fcondition2}, we have
\begin{align*}
  &\Big( \int_{|x-y|\geq R}{\Big||F(u_n)|\chi_{|u_n|\leq K_\varepsilon}-|F(u)|\chi_{|u|\leq K_\varepsilon}\Big|^{\frac{8-\mu}{8-2\mu}}}dy\Big)^{\frac{8-2\mu}{8-\mu}}\\\leq& C\Big(\|u_n\|_{\frac{(8-\mu)(\nu+1)}{8-2\mu}}^{\nu+1}+\|u_n\|_{\frac{(8-\mu)(q+1)}{8-2\mu}}^{q+1}+\|u\|_{\frac{(8-\mu)(\nu+1)}{8-2\mu}}^{\nu+1}+\|u\|_{\frac{(8-\mu)(q+1)}{8-2\mu}}^{q+1}\Big)\leq C.
\end{align*}
Choosing $R>0$ large enough, then we have
\begin{equation}\label{w1}
 \varpi_n(x)\rightarrow \varpi(x) \quad \text{for any $x\in \mathbb{R}^4$}.
\end{equation}
From \eqref{fcondition2}, for any $x\in \mathbb{R}^4$, one also has
\begin{align}\label{w2}
  \varpi_n(x)&\leq \Big( \int_{|x-y|\leq R}{\Big||F(u_n)|\chi_{|u_n|\leq K_\varepsilon}\Big|^{\frac{4+\mu}{4-\mu}}}dy\Big)^{\frac{4-\mu}{4+\mu}}\Big(\int_{|x-y| \leq R}|x-y|^{-\frac{4+\mu}{2}}dy\Big)^{\frac{2\mu}{4+\mu}}\nonumber\\
  &\quad+\Big( \int_{|x-y|\geq R}{\Big||F(u_n)|\chi_{|u_n|\leq K_\varepsilon}\Big|^{\frac{8-\mu}{8-2\mu}}}dy\Big)^{\frac{8-2\mu}{8-\mu}}\Big(\int_{|x-y| \geq R}|x-y|^{\mu-8}dy\Big)^{\frac{\mu}{8-\mu}}\nonumber\\
  &\leq \Big(\frac{4\pi^2}{4-\mu}R^{\frac{4-\mu}{2}}\Big)^{\frac{2\mu}{4+\mu}}\Big(\frac{\pi^2R^4}{2}\Big)^{\frac{4-\mu}{4+\mu}}\max\limits_{|t|\leq K_\varepsilon}|F(t)|\nonumber\\&
  \quad+
  C\Big(\frac{2\pi^2}{(4-\mu)R^{4-\mu}}\Big)^{\frac{\mu}{8-\mu}}\Big(\|u_n\|_{\frac{(8-\mu)(\nu+1)}{8-2\mu}}^{\nu+1}+\|u_n\|_{\frac{(8-\mu)(q+1)}{8-2\mu}}^{q+1}\Big)\leq C.
\end{align}
Then
\begin{equation*}
  |\varpi_n(x)f(u_n(x))\chi_{|u_n| \leq M_\varepsilon}\phi(x)|\leq C \|\phi\|_\infty\max\limits_{|t|\leq M_\varepsilon}|f(t)|,\quad \text{for any $x\in \Omega$}.
\end{equation*}
This together with \eqref{w1}, using the Lebesgue dominated convergence theorem, yields that
\begin{equation}\label{final}
  \lim\limits_{n\rightarrow\infty}\int\limits_{\{|u_n|\leq M_\varepsilon\}\cap \{|u|\neq M_\varepsilon\}} (I_\mu*(F(u_n)\chi_{|u_n|\leq K_\varepsilon}))|f(u_n)\phi| dx=\int\limits_{|u|\leq M_\varepsilon} (I_\mu*(F(u)\chi_{|u|\leq K_\varepsilon}))|f(u)\phi| dx.
\end{equation}
By the arbitrariness of $\varepsilon>0$, from \eqref{begin1}, \eqref{begin2}, \eqref{give2}, \eqref{give3} and \eqref{final}, we get \eqref{strong4}.
\end{proof}

\begin{lemma}\label{Pohozaev}
If $u\in H^2(\mathbb{R}^4)$ is a critical point of $\mathcal J(u)$ on $S(c)$, then $u\in \mathcal{P}(c)$.
\end{lemma}
\begin{proof}
If $u\in H^2(\mathbb{R}^4)$ is a critical point of $\mathcal J(u)$ on $S(c)$, there is $\lambda\in\mathbb{R}$ such that
\begin{equation}\label{laeq}
 \mathcal J'(u)-\lambda u=0\ \ \mbox{in}\  (H^{2}(\mathbb{R}^4))^*.
\end{equation}
Testing \eqref{laeq} with $u$, we have
\begin{equation}\label{laeq1}
  \int_{\mathbb{R}^4}|\Delta u |^2\ dx+\beta\int_{\mathbb{R}^4}|\nabla u|^2 dx-\int_{\mathbb{R}^4} (I_\mu*F(u))f(u)u dx-\lambda \int_{\mathbb{R}^4} |u|^2 dx=0.
\end{equation}
On the other hand, consider a cut-off function $\varphi\in C_0^\infty(\mathbb{R}^4,[0,1])$ such that $\varphi(x)=1$ if $|x|\leq1$, $\varphi(x)=0$ if $|x|\geq2$.
For any fixed $\rho>0$, set $\widetilde{u}_\rho(x)=\varphi(\rho x)x\cdot\nabla u(x)$ as a test function of \eqref{abs1} to obtain
\begin{equation*}
  \int_{\mathbb{R}^4}\Delta u \Delta\widetilde{u}_\rho dx+ \beta\int_{\mathbb{R}^4}\nabla u \cdot\nabla \widetilde{u}_\rho dx-\int_{\mathbb{R}^4} (I_\mu*F(u))f(u)\widetilde{u}_\rho dx-\lambda \int_{\mathbb{R}^4} u\widetilde{u}_\rho dx=0.
\end{equation*}
By \cite{MS2}, we know
\begin{equation*}
  \lim_{\rho\rightarrow0}\int_{\mathbb{R}^4}\nabla u\cdot \nabla \widetilde{u}_\rho dx=-\int_{\mathbb{R}^4}|\nabla u|^2 dx,\quad\lim_{\rho\rightarrow0}\int_{\mathbb{R}^4} u\widetilde{u}_\rho dx=-2\int_{\mathbb{R}^4}| u|^2 dx,
\end{equation*}
and
\begin{equation*}
 \lim_{\rho\rightarrow0} \int_{\mathbb{R}^4} (I_\mu*F(u))f(u)\widetilde{u}_\rho dx=-\frac{8-\mu}{2} \int_{\mathbb{R}^4} (I_\mu*F(u))F(u)dx.
\end{equation*}
Integrating by parts, we find that
\begin{align*}
  \int_{\mathbb{R}^4}\Delta u \Delta\widetilde{u}_\rho dx&=2\int_{\mathbb{R}^4}\varphi(\rho x)|\Delta u|^2 dx+\int_{\mathbb{R}^4}\varphi(\rho x)\Delta u(x\cdot \nabla (\Delta u))dx\\
  &=2\int_{\mathbb{R}^4}\varphi(\rho x)|\Delta u|^2 dx+\int_{\mathbb{R}^4}\varphi(\rho x)(x\cdot \nabla (\frac{|\Delta u|^2}{2}))dx
  \\&=2\int_{\mathbb{R}^4}\varphi(\rho x)|\Delta u|^2 dx-\int_{\mathbb{R}^4}\rho x\cdot\nabla \varphi(\rho x)\frac{|\Delta u|^2}{2}dx-\int_{\mathbb{R}^4}4\varphi(\rho x)\frac{|\Delta u|^2}{2}dx.
\end{align*}
The Lebesgue dominated convergence theorem implies that
\begin{equation*}
  \lim_{\rho\rightarrow0} \int_{\mathbb{R}^4}\Delta u \Delta\widetilde{u}_\rho dx=0.
\end{equation*}
Thus
\begin{equation}\label{po1}
  -\beta\int_{\mathbb{R}^4}|\nabla u|^2 dx=-2\lambda\int_{\mathbb{R}^4} |u|^2 dx-\frac{8-\mu}{2} \int_{\mathbb{R}^4} (I_\mu*F(u))F(u)dx.
\end{equation}
From \eqref{laeq1} and \eqref{po1}, we deduce that
 \begin{align*}
  &2\int_{\mathbb{R}^4}|\Delta u |^2\ dx+\beta\int_{\mathbb{R}^4}|\nabla u|^2 dx+\frac{8-\mu}{2} \int_{\mathbb{R}^4} (I_\mu*F(u))F(u)dx-2\int_{\mathbb{R}^4} (I_\mu*F(u))f(u)u dx=0.
\end{align*}
That is, $u\in \mathcal{P}(c)$.
\end{proof}

\section{The estimation for the upper bound of $E(c)$}\label{estimation}

In this section, we introduce the modified Adams functions and estimate the upper bound of $E(c)$.

\begin{lemma}\label{minimax}
Assume that $(f_1)-(f_3)$ hold, then for any fixed $u\in S(c)$, we have

$(i)$ $\mathcal{J}(\mathcal{H}(u,s))\rightarrow0^+$ as $s\rightarrow-\infty$;

$(ii)$ $\mathcal{J}(\mathcal{H}(u,s))\rightarrow-\infty$ as $s\rightarrow+\infty$.
\end{lemma}
\begin{proof} A straightforward calculation shows that for any $q>2$,
\begin{equation*}
 \|\mathcal{H}(u,s)\|_2=c,\,\,\, \|\Delta \mathcal{H}(u,s)\|_2=e^{2s} \|\Delta u\|_2,\,\,\, \|\nabla \mathcal{H}(u,s)\|_2=e^{s} \|\nabla u\|_2,\,\,\,\|\mathcal{H}(u,s)\|_q=e^{\frac{2(q-2)s}{q}}\|u\|_q.
\end{equation*}
So there exists $s_1<<0$ such that
\begin{equation*}
  \|\Delta \mathcal{H}(u,s)\|_2^2= e^{4s} \|\Delta u\|_2^2<\frac{8-\mu}{8},\quad \text{for all $s\leq s_1$}.
\end{equation*}
Fix $\alpha>32\pi^2$ close to $32\pi^2$ and $m>1$ close to $1$ with  $\frac{8\alpha m \|\Delta \mathcal{H}(u,s)\|^2}{8-\mu}\leq 32\pi^2$ for all $s\leq s_1$. For $m'=\frac{m}{m-1}$, using Proposition \ref{HLS}, Lemma \ref{adams}, the H\"{o}lder inequality and the Sobolev inequality, we have
\begin{align}\label{close1}
 &\int_{\mathbb{R}^4}(I_\mu*F(\mathcal{H}(u,s)))F(\mathcal{H}(u,s))dx\leq C\|F(\mathcal{H}(u,s))\|_{\frac{8}{8-\mu}}^2\nonumber\\
\leq & C\bigg(\int_{\mathbb{R}^4}\Big(\xi|\mathcal{H}(u,s)|^{\nu+1}+C_\xi|\mathcal{H}(u,s)|^{q+1}(e^{\alpha \mathcal{H}^2(u,s)}-1)\Big)^{\frac{8}{8-\mu}}dx\bigg)^{\frac{8-\mu}{4}}\nonumber\\
\leq& C \bigg(\int_{\mathbb{R}^4}\Big(|\mathcal{H}(u,s)|^{\frac{8(\nu+1)}{8-\mu}}+|\mathcal{H}(u,s)|^{\frac{8(q+1)}{8-\mu}}(e^{\frac{8\alpha \mathcal{H}^2(u,s)}{8-\mu}}-1)\Big)dx\bigg)^{\frac{8-\mu}{4}}\nonumber\\
\leq& C\|\mathcal{H}(u,s)\|_{\frac{8(\nu+1)}{8-\mu}}^{2(\nu+1)}+C\bigg(\int_{\mathbb{R}^4}|\mathcal{H}(u,s)|^{\frac{8(q+1)}{8-\mu}}(e^{\frac{8\alpha \mathcal{H}^2(u,s)}{8-\mu}}-1)dx\bigg)^{\frac{8-\mu}{4}}\nonumber\\
\leq&  C\|\mathcal{H}(u,s)\|_{\frac{8(\nu+1)}{8-\mu}}^{2(\nu+1)}+C\Big(\int_{\mathbb{R}^4}|\mathcal{H}(u,s)|^{\frac{8(q+1)m'}{8-\mu}}dx\Big)^{\frac{8-\mu}{4m'}}\Big(\int_{\mathbb{R}^4}(e^{\frac{8\alpha m \|\Delta \mathcal{H}(u,s)\|^2}{8-\mu}(\frac{\mathcal{H}(u,s)}{\|\Delta\mathcal{H}(u,s)\|})^2}-1)dx\Big)^{\frac{8-\mu}{4m}}\nonumber\\
\leq& C\|\mathcal{H}(u,s)\|_{\frac{8(\nu+1)}{8-\mu}}^{2(\nu+1)}+C\|\mathcal{H}(u,s)\|_{\frac{8(q+1)m'}{8-\mu}}^{2(q+1)}\nonumber\\
=&Ce^{(4\nu+\mu-4)s}\|u\|_{\frac{8(\nu+1)}{8-\mu}}^{2(\nu+1)}+Ce^{(4q+4-\frac{8-\mu}{m'})s}\|u\|_{\frac{8(q+1)m'}{8-\mu}}^{2(q+1)},\quad \text{for all $s\leq s_1$}.
\end{align}
Since $\beta\geq0$, $\nu>2-\frac{\mu}{4}$, $q>2$, and $m'$ large enough, it follows from \eqref{close1} that
\begin{equation*}
  \mathcal{J}(\mathcal{H}(u,s))\geq \frac{e^{4s}}{2}\|\Delta u\|_2^2-Ce^{(4\nu+\mu-4)s}\|u\|_{\frac{8(\nu+1)}{8-\mu}}^{2(\nu+1)}-Ce^{(4q+4-\frac{8-\mu}{m'})s}\|u\|_{\frac{8(q+1)m'}{8-\mu}}^{2(q+1)}\rightarrow 0^+, \quad\text{as $s\rightarrow-\infty$.}
\end{equation*}
For any fixed $s>>0$,
set
\begin{equation*}
 \mathcal{M}(t)=\frac{1}{2}\int_{\mathbb{R}^4}(I_\mu*F(tu))F(tu)dx,\quad\text{for $t>0$}.
\end{equation*}
It follows from $(f_3)$ that
\begin{equation*}
 \frac{\frac{d\mathcal{M}(t)}{dt}}{\mathcal{M}(t)}>\frac{2\theta}{t},\quad\text{for $t>0$}.
\end{equation*}
Thus, integrating this over $[1,e^{2s}]$, we have
\begin{equation}\label{imp}
  \frac{1}{2}\int_{\mathbb{R}^4}(I_\mu*F(e^{2s}u))F(e^{2s}u)dx\geq \frac{e^{4\theta s}}{2}\int_{\mathbb{R}^4}(I_\mu*F(u))F(u)dx.
\end{equation}
Therefore,
\begin{align*}
\mathcal J(\mathcal{H}(u,s))
\leq&
\frac{e^{4s}}{2}\int_{\mathbb{R}^4}|\Delta u|^2dx+\frac{\beta e^{2s}}{2}\int_{\mathbb{R}^4}|\nabla u|^2dx-\frac{e^{(4\theta+\mu-8) s}}{2}\int_{\mathbb{R}^4}(I_\mu*F(u)F(u)dx.
\end{align*}
Since $\theta>3-\frac{\mu}{4}$, the above inequality yields that $\mathcal{J}(\mathcal{H}(u,s))\rightarrow-\infty$ as $s\rightarrow+\infty$.
\end{proof}
\begin{lemma}\label{equi}
Assume that $f$ satisfies $(f_1)-(f_3)$, if $(f_6)$ or $(f_7)$ holds, then for any fixed $u\in S(c)$,
the following statements hold.

$(i)$ The function $g_u(s)={\mathcal J}(\mathcal{H}(u,s))$ achieves its maximum with positive level at a unique point $s_u\in \mathbb{R}$ such that $\mathcal{H}(u,s_u) \in \mathcal{P}(c)$.

$(ii)$ The mapping $u\mapsto s_u$ is continuous in $u\in S(c)$.

$(iii)$ $s_{-u}=s_u$.
\end{lemma}
\begin{proof}
$(i)$ The proof can be found from \cite[Lemma 3.9]{CW}, we give it here for the completeness.
From Lemma \ref{minimax},
there exists $s_u\in \mathbb{R}$ such that $P(\mathcal{H}(u,s_u))=\frac{d}{ds}\mathcal{J}(\mathcal{H}(u,s))|_{s=s_u}=0$, and $\mathcal{J}(\mathcal{H}(u,s_u))>0$.
In the following, we prove the uniqueness of $s_u$ for any $u\in S(c)$.

{\bf Case 1:} If $(f_1)-(f_3)$  and $(f_6)$ hold. Taking into account that $\frac{d}{ds}\mathcal{J}(\mathcal{H}(u,s))|_{s=s_u}=0$, using $\beta\geq0$ and $(f_6)$, we deduce that
\begin{align*}
  \frac{d^2}{ds^2}\mathcal{J}(\mathcal{H}(u,s))\Big|_{s=s_u}=&
8e^{4s_u}\int_{\mathbb{R}^4}|\Delta u|^2dx+2\beta e^{2s_u}\int_{\mathbb{R}^4}|\nabla u|^2dx
\\&-\frac{(8-\mu)^2}{2}e^{(\mu-8)s_u}\int_{\mathbb{R}^4}(I_\mu*F(e^{2s_u}u))F(e^{2s_u}u)dx
\\&+(28-4\mu)e^{(\mu-8)s_u}\int_{\mathbb{R}^4}(I_\mu*F(e^{2s_u}u))f(e^{2s_u}u)e^{2s_u}udx\\
&-4e^{(\mu-8)s_u}\int_{\mathbb{R}^4}(I_\mu*f(e^{2s_u}u)e^{2s_u}u)f(e^{2s_u}u)e^{2s_u}udx\\&
-4e^{(\mu-8)s_u}\int_{\mathbb{R}^4}(I_\mu*F(e^{2s_u}u))f'(e^{2s_u}u)e^{4s_u}u^2dx\\
=&-2\beta \int_{\mathbb{R}^4}|\nabla \mathcal{H}(u,s_u)|^2dx-(8-\mu)(6-\frac{\mu}{2})\int_{\mathbb{R}^4}(I_\mu*F(\mathcal{H}(u,s_u)))F(\mathcal{H}(u,s_u))dx\\
&+(36-4\mu)\int_{\mathbb{R}^4}(I_\mu*F(\mathcal{H}(u,s_u)))f(\mathcal{H}(u,s_u))\mathcal{H}(u,s_u)dx\\
&-4\int_{\mathbb{R}^4}(I_\mu*f(\mathcal{H}(u,s_u))\mathcal{H}(u,s_u))f(\mathcal{H}(u,s_u))\mathcal{H}(u,s_u)dx\\&
-4\int_{\mathbb{R}^4}(I_\mu*F(\mathcal{H}(u,s_u)))f'(\mathcal{H}(u,s_u))\mathcal{H}^2(u,s_u)dx \\
=&-2\beta \int_{\mathbb{R}^4}|\nabla \mathcal{H}(u,s_u)|^2dx +4\int_{\mathbb{R}^4}\int_{\mathbb{R}^4}\frac{A}{|x-y|^\mu}dxdy<0,
\end{align*}
where
\begin{align*}
  A=&(3-\frac{\mu}{4})F(\mathcal{H}(u(y),s_{u(y)}))\overline{F}(\mathcal{H}(u(x),s_{u(x)}))-F(\mathcal{H}(u(y),s_{u(y)}))\overline{F}'(\mathcal{H}(u(x),s_{u(x)}))\mathcal{H}(u(x),s_{u(x)})\\&
  -\overline{F}(\mathcal{H}(u(y),s_{u(y)}))(\overline{F}(\mathcal{H}(u(x),s_{u(x)}))-F(\mathcal{H}(u(x),s_{u(x)})))<0,
\end{align*}
this prove the uniqueness of $s_u$.

{\bf Case 2:} If $(f_1)-(f_3)$  and $(f_7)$ hold. Since
\begin{align*}
P(\mathcal{H}(u,s))=&e^{4s}\Big[2 \int_{\mathbb{R}^4}|\Delta u|^2dx+ \frac{\beta}{e^{2s}}\int_{\mathbb{R}^4}|\nabla u|^2dx -2\int_{\mathbb{R}^4}\Big(I_\mu*\frac{F(e^{2s}u)}{(e^{2s})^{3-\frac{\mu}{4}}}\Big)\frac{\overline{F}(e^{2s}u)}{(e^{2s})^{3-\frac{\mu}{4}}}\Big]dx.
\end{align*}
Denote
\begin{equation*}
\psi(s)=\int_{\mathbb{R}^4}\Big(I_\mu*\frac{F(e^{2s}u)}{(e^{2s})^{3-\frac{\mu}{4}}}\Big)\frac{\overline{F}(e^{2s}u)}{(e^{2s})^{3-\frac{\mu}{4}}}dx.
\end{equation*}
For any $t\in \mathbb{R}\backslash \{0\}$, from $(f_3)$ and $(f_7)$, we see that $\frac{F(st)}{s^{3-\frac{\mu}{4}}}$ is increasing in $s\in (0,+\infty)$ and $\frac{\overline{F}(st)}{s^{3-\frac{\mu}{4}}}$ is non-decreasing in $s\in (0,+\infty)$. This fact implies $\psi(s)$ is increasing in $s\in \mathbb{R}$ and there is at most one $s_u\in \mathbb{R}$ such that $\mathcal{H}(u,s_u) \in \mathcal{P}(c)$.

$(ii)$ By $(i)$, the mapping $u\mapsto s_u$ is well defined. Let $\{u_n\}\subset S(c)$ be any sequence such that $u_n\rightarrow u$ in $H^2(\mathbb{R}^4)$ as $n\rightarrow\infty$. We only need to prove that up to a subsequence, $s_{u_n}\rightarrow s_u$ in $\mathbb{R}$ as $n\rightarrow\infty$.

We first show that $\{s_{u_n}\}$ is bounded. If up to a subsequence, $s_{u_n}\rightarrow+\infty$ as $n\rightarrow\infty$, then by \eqref{imp} and $u_n\rightarrow u\neq0$ in $ H^2(\mathbb{R}^4)$ as $n\rightarrow\infty$, we have
\begin{align*}
  0&\leq \lim\limits_{n\rightarrow\infty}e^{-4s_{u_n}}\mathcal J(\mathcal{H}(u_n,s_{u_n}))\\
  &\leq \lim\limits_{n\rightarrow\infty}\frac{1}{2}\Big[\int_{\mathbb{R}^4}|\Delta u_n|^2dx+\frac{\beta}{e^{2s_{u_n}}}\int_{\mathbb{R}^4}|\nabla u_n|^2dx-e^{(4\theta+\mu-12)s_{u_n}}\int_{\mathbb{R}^4}(I_\mu*F(u_n))F(u_n)dx\Big]=-\infty,
\end{align*}
which is a contradiction. Therefore, $\{s_{u_n}\}$ is bounded from above. On the other hand, by Lemma \ref{biancon}, $\mathcal{H}(u_n,s_{u})\rightarrow \mathcal{H}(u,s_{u})$ in $H^2(\mathbb{R}^4)$ as $n\rightarrow\infty$, it follows from $(i)$ that
\begin{equation*}
  \mathcal J(\mathcal{H}(u_n,s_{u_n}))\geq \mathcal J(\mathcal{H}(u_n,s_{u}))=\mathcal J(\mathcal{H}(u,s_{u}))+o_n(1),
\end{equation*}
and thus
\begin{equation*}
  \liminf\limits_{n\rightarrow\infty}\mathcal J(\mathcal{H}(u_n,s_{u_n}))\geq \mathcal J(\mathcal{H}(u,s_{u}))>0.
\end{equation*}
If up to a subsequence, $s_{u_n}\rightarrow-\infty$ as $n\rightarrow\infty$, using $(f_3)$, we get
\begin{equation*}
  \mathcal J(\mathcal{H}(u_n,s_{u_n}))\leq \frac{e^{4s_{u_n}}}{2}\|\Delta u\|_2^2+\frac{\beta e^{2s_{u_n}}}{2}\|\nabla u\|_2^2\rightarrow0, \quad\text{as}\,\,\, n\rightarrow\infty,
\end{equation*}
which is an absurd. Therefore, $\{s_{u_n}\}$ is bounded from below.
Up to a subsequence, we assume that $s_{u_n}\rightarrow s_*$ as $n\rightarrow\infty$. Recalling that $u_n\rightarrow u$ in $H^2(\mathbb{R}^4)$ as $n\rightarrow\infty$, then $\mathcal{H}(u_n,s_{u_n})\rightarrow \mathcal{H}(u,s_{*})$ in $H^2(\mathbb{R}^4)$ as $n\rightarrow\infty$. Since $P(\mathcal{H}(u_n,s_{u_n}))=0$ for any $n\in \mathbb{N}^+$, it follows that $P(\mathcal{H}(u,s_{*}))=0$. By the uniqueness of $s_u$, we get $s_u=s_*$, thus $(ii)$ is proved.

$(iii)$ Since $f$ is odd or even in $\mathbb{R}$, it's easy to find that $\int_{\mathbb{R}^4}(I_\mu*F(u))F(u)dx$ and $\int_{\mathbb{R}^4}(I_\mu*F(u))f(u)udx$ are even in $H^2(\mathbb{R}^4)$, thus
\begin{equation*}
  P(\mathcal{H}(-u,s_u))=P(-\mathcal{H}(u,s_u))=P(\mathcal{H}(u,s_u)),
\end{equation*}
which implies that $s_{-u}=s_u$.
\end{proof}
\begin{lemma}\label{continuous}
Assume that $(f_1)-(f_3)$ hold, then there exists $\delta>0$ small enough such that
\begin{equation*}
  \mathcal  J(u)
\geq\frac{1}{4}\int_{\mathbb{R}^4}|\Delta u|^2dx+\frac{\beta}{2}\int_{\mathbb{R}^4}|\nabla u|^2dx\quad\text{and}\quad P(u)
\geq\int_{\mathbb{R}^4}|\Delta u|^2dx+\beta\int_{\mathbb{R}^4}|\nabla u|^2dx
\end{equation*}
for all $u\in S(c)$ satisfying $\|\Delta u\|_2\leq \delta$.
\end{lemma}
\begin{proof}
If $\delta<\sqrt{\frac{8-\mu}{8}}$, then $\|\Delta u\|_2^2<\frac{8-\mu}{8}$. Fix $\alpha>32\pi^2$ close to $32\pi^2$ and $m>1$ close to $1$ such that
\begin{equation*}
  \frac{8\alpha m\| \Delta u\|_2^2}{8-\mu}\leq 32\pi^2.
\end{equation*}
Arguing as \eqref{close1}, for $m'=\frac{m}{m-1}$, by \eqref{GNinequality}, we have
\begin{align*}
  \int_{\mathbb{R}^4}(I_\mu*F(u))F(u)dx&\leq C\|u\|_{\frac{8(\nu+1)}{8-\mu}}^{2(\nu+1)}+C\|u\|_{\frac{8(q+1)m'}{8-\mu}}^{2(q+1)}\\&\leq Cc^{4-\frac{\mu}{2}}\|\Delta u\|_2^{2\nu+\frac{\mu}{2}-2}+C c^{\frac{8-\mu}{2m'}} \|\Delta u\|_2^{2(q+1)-\frac{8-\mu}{2m'}}\\
  &\leq\Big [Cc^{4-\frac{\mu}{2}}\delta^{2\nu+\frac{\mu}{2}-4}+C c^{\frac{8-\mu}{2m'}} \delta^{2q-\frac{8-\mu}{2m'}}\Big]\|\Delta u\|_2^2
\end{align*}
and \begin{align*}
  \int_{\mathbb{R}^4}(I_\mu*F(u))f(u)udx&\leq C\|u\|_{\frac{8(\nu+1)}{8-\mu}}^{2(\nu+1)}+C\|u\|_{\frac{8(q+1)m'}{8-\mu}}^{2(q+1)}
  \\&\leq Cc^{4-\frac{\mu}{2}}\|\Delta u\|_2^{2\nu+\frac{\mu}{2}-2}+C c^{\frac{8-\mu}{2m'}} \|\Delta u\|_2^{2(q+1)-\frac{8-\mu}{2m'}}\\
  &\leq\Big [Cc^{4-\frac{\mu}{2}}\delta^{2\nu+\frac{\mu}{2}-4}+C c^{\frac{8-\mu}{2m'}} \delta^{2q-\frac{8-\mu}{2m'}}\Big]\|\Delta u\|_2^2.
\end{align*}
Since $\nu>2-\frac{\mu}{4}$, $q>2$ and $m'$ large enough, choosing $\delta<\sqrt{\frac{8-\mu}{8}}$ small enough, we conclude the result.
\end{proof}

\begin{lemma}\label{inf}
Assume that $f$ satisfies $(f_1)-(f_3)$, if $(f_6)$ or $(f_7)$ holds, then we have $\inf\limits_{u\in \mathcal{P}(c)}\|\Delta u\|_2>0$ and $E(c)>0$.
\end{lemma}
\begin{proof}
By Lemma \ref{equi}, we obtain $\mathcal{P}(c)\neq\emptyset$. Suppose that there exists a sequence $\{u_n\}\subset \mathcal{P}(c)$ such that $\|\Delta u_n\|_2\rightarrow0$ as $n\rightarrow\infty$, then by Lemma \ref{continuous}, up to a subsequence,
\begin{equation*}
  0=P(u_n)\geq\int_{\mathbb{R}^4}|\Delta u_n|^2dx+\beta\int_{\mathbb{R}^4}|\nabla u_n|^2dx\geq0,
\end{equation*}
which implies that $\int_{\mathbb{R}^4}|\Delta u_n|^2dx=\beta\int_{\mathbb{R}^4}|\nabla u_n|^2dx=0$ for any $n\in \mathbb{N}^+$.
Hence, by $(f_3)$ and $P(u_n)=0$, we have
\begin{equation*}
  0=\int_{\mathbb{R}^4}(I_\mu*F(u_n))\Big(\frac{8-\mu}{2}F(u_n)-2f(u_n)u_n\Big)dx\leq \int_{\mathbb{R}^4}(I_\mu*F(u_n))\Big(\frac{8-\mu}{2\theta}-2\Big)f(u_n)u_ndx\leq0.
\end{equation*}
So $u_n\rightarrow0$ a.e. in $\mathbb{R}^4$, which is contradict to $c>0$.

For any $u\in \mathcal{P}(c)$, by Lemma \ref{equi}, we have
\begin{equation*}
  \mathcal{J}(u)=\mathcal{J}(\mathcal{H}(u,0))\geq \mathcal{J}(\mathcal{H}(u,s))\quad\text{for all $s\in \mathbb{R}$}.
\end{equation*}
Let $\delta>0$ be the number given by Lemma \ref{continuous} and $2s:=\ln\frac{\delta}{\|\Delta u\|_2}$. Then $\|\Delta(\mathcal{H}(u,s))\|_2=\delta$, by $\beta\geq0$ and Lemma \ref{continuous}, we deduce that
\begin{equation*}
  \mathcal{J}(u)\geq \mathcal{J}(\mathcal{H}(u,s))\geq \frac{1}{4}\int_{\mathbb{R}^4}|\Delta \mathcal{H}(u,s)|^2dx+\frac{\beta}{2}\int_{\mathbb{R}^4}|\nabla \mathcal{H}(u,s)|^2dx\geq\frac{\delta^2}{4}>0.
\end{equation*}
By the arbitrariness of $u\in \mathcal{P}(c)$, we derive the conclusion.
\end{proof}

\noindent{\bf The modified Adams functions}

In order to estimate the upper bound of $E(c)$, let us introduce the modified Adams functions as follows.
For any $\varphi(t)\in C_0^\infty([0,\infty),[0,1])$ such that $\varphi(t)=1$ if $0\leq t\leq1$, $\varphi(t)=0$ if $t\geq2$. Define a sequence of functions $\tilde{\omega}_n$ by
\begin{align*}
  \begin{split}
 \tilde{\omega}_n(x)=\left\{
  \begin{array}{ll}
  \sqrt{\frac{\log n}{8\pi^2}}+\frac{1-n^2|x|^2}{\sqrt{32\pi^2\log n}},&\quad \text {for} \,\,\,|x|\leq \frac{1}{n},\\
  -\frac{\log |x|}{\sqrt{8\pi^2\log n}},&\quad \text {for} \,\,\,\frac{1}{n}<|x|\leq1,\\
  -\frac{\varphi(|x|)\log |x|}{\sqrt{8\pi^2\log n}},&\quad \text {for} \,\,\,1<|x|<2,\\
  0,&\quad \text {for} \,\,\,|x|\geq2.
    \end{array}
    \right.
  \end{split}
  \end{align*}
One can check that $\tilde{\omega}_n\in H^2(\mathbb{R}^4)$. A straightforward calculation shows that
\begin{equation*}
  \|\tilde{\omega}_n\|_2^2=\frac{1+32M_1}{128\log n}-\frac{1}{96 n^4}-\frac{1}{192 n^4 \log n}=\frac{1+32M_1}{128\log n}+o(\frac{1}{\log ^4n}),
\end{equation*}
\begin{equation*}
  \|\nabla\tilde{\omega}_n\|_2^2=\frac{1+2M_2}{8\log n}-\frac{1}{12n^2 \log n}=\frac{1+2M_2}{8\log n}+o(\frac{1}{\log ^3n}),
\end{equation*}
and
\begin{equation*}
  \|\Delta\tilde{\omega}_n\|_2^2=1+\frac{4+M_3}{4\log n},
\end{equation*}
where
 \begin{equation*}
 M_1=\int_1^2\varphi^2(r)r^3\log ^2rdr,
\end{equation*}
 \begin{equation*}
 M_2=\int_1^2\Big(\varphi'(r)\log r+\frac{\varphi(r)}{r}\Big)^2r^3dr,
\end{equation*}
and
\begin{equation*}
M_3=\int_1^2\Big(\varphi''(r)\log r+\frac{-\varphi(r)+3\varphi'(r)+2r\varphi'(r)+3r\varphi'(r)\log r}{r^2}\Big)r^3dr.
\end{equation*}
Let $\omega_n=\frac{c\tilde{\omega}_n}{\|\tilde{\omega}_n\|_2}$. Then $\omega_n\in S(c)$ and
 \begin{equation}\label{t omega}
 \|\nabla \omega_n\|_2^2=\frac{c^2(\frac{1+2M_2}{8\log n}+o(\frac{1}{\log ^3n}))}{\frac{1+32M_1}{128\log n}+o(\frac{1}{\log ^4n})}=\frac{16c^2(1+2M_2)}{1+32M_1}\Big(1+o(\frac{1}{\log ^2n})\Big),
\end{equation}
 \begin{equation}\label{l omega}
 \|\Delta \omega_n\|_2^2=\frac{c^2(1+\frac{4+M_3}{4\log n})}{\frac{1+32M_1}{128\log n}+o(\frac{1}{\log ^4n})}=\frac{128c^2}{1+32M_1}\Big(\frac{4+M_3}{4}+\log n+o(\frac{1}{\log ^2n})\Big).
\end{equation}
Furthermore, we have
\begin{align}\label{deomega}
  \begin{split}
 \omega_n(x)=\left\{
  \begin{array}{ll}
 \frac{c(1+o(\frac{1}{\log ^3n}))}{\sqrt{\frac{1+32M_1}{128}}}\Big( \frac{\log n}{\sqrt{8\pi^2}}+\frac{1-n^2|x|^2}{\sqrt{32\pi^2}}\Big),&\quad \text {for} \,\,\,|x|\leq \frac{1}{n},\\
  -\frac{c(1+o(\frac{1}{\log ^3n}))}{\sqrt{\frac{1+32M_1}{128}}}\frac{\log |x|}{\sqrt{8\pi^2}},&\quad \text {for} \,\,\,\frac{1}{n}<|x|\leq1,\\
  -\frac{c(1+o(\frac{1}{\log ^3n}))}{\sqrt{\frac{1+32M_1}{128}}}\frac{\varphi(|x|)\log |x|}{\sqrt{8\pi^2}},&\quad \text {for} \,\,\,1<|x|<2,\\
  0,&\quad \text {for} \,\,\,|x|\geq2.
    \end{array}
    \right.
  \end{split}
  \end{align}
For any $t>0$, let
\begin{equation*}
  g_n(t):=
\mathcal J(t\omega_n(t^{\frac{1}{2}}x))=\frac{ t^2}{2}\int_{\mathbb{R}^4}|\Delta \omega_n|^2dx+\frac{\beta t}{2}\int_{\mathbb{R}^4}|\nabla \omega_n|^2dx-\frac{t^{\frac{\mu}{2}-4}}{2}\int_{\mathbb{R}^4}(I_\mu*F(t\omega_n))F(t\omega_n)dx.
\end{equation*}
By Lemmas \ref{equi} and \ref{inf}, we know $E(c)=\inf\limits_{u\in S(c)}\max\limits_{s\in \mathbb{R}}\mathcal J(\mathcal{H}(u,s))>0$, this together with $\omega_n\in S(c)$ yields that
\begin{equation*}
  0<E(c)\leq\max\limits_{s\in \mathbb{R}}\mathcal J(\mathcal{H}(\omega_n,s))=\max\limits_{t>0}g_n(t).
\end{equation*}
\begin{lemma}\label{attain}
Assume that $(f_1)-(f_3)$ hold, then for any fixed $n\in \mathbb{N}^+$, $\max\limits_{t>0}g_n(t)$ is attained at some $t_n>0$.
\end{lemma}
\begin{proof}
For any fixed $n\in \mathbb{N}^+$, as $t>0$ small, fix $\alpha>32\pi^2$ close to $32\pi^2$ and $m>1$ close to $1$ such that
\begin{equation*}
  \frac{8\alpha m\| \Delta(t\omega_n)\|_2^2}{8-\mu}\leq 32\pi^2.
\end{equation*}
Arguing as \eqref{close1}, for $m'=\frac{m}{m-1}$, we have
\begin{align*}
  \frac{t^{\frac{\mu}{2}-4}}{2}\int_{\mathbb{R}^4}(I_\mu*F(t\omega_n))F(t\omega_n)dx&\leq 
 Ct^{\frac{\mu}{2}-4}\|t\omega_n\|_{\frac{8(\nu+1)}{8-\mu}}^{2(\nu+1)}+Ct^{\frac{\mu}{2}-4}\|t\omega_n\|_{\frac{8(q+1)m'}{8-\mu}}^{2(q+1)}\\
  &=Ct^{2\nu-2+\frac{\mu}{2}}\|\omega_n\|_{\frac{8(\nu+1)}{8-\mu}}^{2(\nu+1)}+Ct^{2q-2+\frac{\mu}{2}}\|\omega_n\|_{\frac{8(q+1)m'}{8-\mu}}^{2(q+1)},
\end{align*}
where $\nu>2-\frac{\mu}{4}$ and $q>2$. So $g_n(t)>0$ for $t>0$ small enough.
By \eqref{imp}, we obtain
\begin{equation*}
  \frac{t^{\frac{\mu}{2}-4}}{2}\int_{\mathbb{R}^4}(I_\mu*F(t\omega_n))F(t\omega_n)dx\geq \frac{t^{2\theta +\frac{\mu}{2}-4 }}{2}\int_{\mathbb{R}^4}(I_\mu*F(\omega_n))F(\omega_n)dx.
\end{equation*}
Since $\theta>3-\frac{\mu}{4}$, we have that $g_n(t)<0$ for $t>0$ large enough. Thus $\max\limits_{t>0}g_n(t)$ is attained at some $t_n>0$.
\end{proof}
\begin{lemma}\label{control}
Assume that $(f_1)-(f_3)$, $(f_5)$ hold, then $\max\limits_{t>0}g_n(t)<\frac{8-\mu}{16}$ for $n\in \mathbb{N}^+$ large enough.
\end{lemma}
\begin{proof}
By Lemma \ref{attain}, $\max\limits_{t>0}g_n(t)$ is attained at some $t_n>0$ and thus $g'_n(t_n)=0$. By $(f_3)$,
\begin{align}\label{useuse}
  t_n^2\int_{\mathbb{R}^4}|\Delta \omega_n|^2dx+\frac{\beta }{2}t_n\int_{\mathbb{R}^4}|\nabla \omega_n|^2dx&=\frac{\mu-8}{4}t_n^{\frac{\mu}{2}-4}\int_{\mathbb{R}^4}(I_\mu*F(t_n\omega_n))F(t_n\omega_n)dx\nonumber\\
  &\quad+t_n^{\frac{\mu}{2}-4}\int_{\mathbb{R}^4}(I_\mu*F(t_n\omega_n))f(t_n\omega_n)t_n\omega_ndx\nonumber\\
  &\geq \frac{4\theta+\mu-8}{4\theta}t_n^{\frac{\mu}{2}-4}\int_{\mathbb{R}^4}(I_\mu*F(t_n\omega_n))f(t_n\omega_n)t_n\omega_ndx.
\end{align}
Note that
\begin{equation*}\label{tF}
  \liminf\limits_{t\rightarrow\infty}\frac{tF(t)}{e^{32\pi^2t^2}}\geq \liminf\limits_{t\rightarrow\infty}\frac{\int_0^tsf(s)ds}{e^{32\pi^2 t^2}}=\liminf\limits_{t\rightarrow\infty}\frac{f(t)}{64\pi^2 e^{32\pi^2 t^2}}.
\end{equation*}
This with $(f_5)$ yields that, for any $\varepsilon>0$. there exists $R_\varepsilon>0$ such that for any $t\geq R_\varepsilon$,
\begin{equation}\label{ftF}
  f(t)\geq (\varrho-\varepsilon)e^{32\pi^2 t^2},\quad tF(t)\geq \frac{\varrho-\varepsilon}{64\pi^2}e^{32\pi^2 t^2}.
\end{equation}

{\bf Case $1$:} $\lim\limits_{n\rightarrow\infty}t_n^2\log n=0$. Then $\lim\limits_{n\rightarrow\infty}t_n=0$, by \eqref{t omega} and \eqref{l omega}, we have that $\frac{ t_n^2}{2}\int_{\mathbb{R}^4}|\Delta \omega_n|^2dx\rightarrow 0$, $\frac{\beta t_n}{2}\int_{\mathbb{R}^4}|\nabla \omega_n|^2dx\rightarrow0$ as $n\rightarrow\infty$. Noting that $F(t_n\omega_n)>0$ by $(f_3)$, so we have
\begin{equation*}
  0<g_n(t_n)\leq\frac{ t_n^2}{2}\int_{\mathbb{R}^4}|\Delta \omega_n|^2dx+\frac{\beta t_n}{2}\int_{\mathbb{R}^4}|\nabla \omega_n|^2dx,
\end{equation*}
thus $\lim\limits_{n\rightarrow\infty}g_n(t_n)=0$, and we conclude.

{\bf Case $2$:} $\lim\limits_{n\rightarrow\infty}t_n^2\log n=l\in (0,\infty]$. We claim that $l<\infty$. Otherwise, if
$l=\infty$, then $\lim\limits_{n\rightarrow\infty}(t_n\log n)=\infty$.
By \eqref{t omega}-\eqref{deomega} and \eqref{useuse}-\eqref{ftF}, we have
\begin{align*}
  &\frac{128c^2t_n^2}{1+32M_1}\Big(\frac{4+M_3}{4}+\log n+o(\frac{1}{\log ^2n})\Big)+\frac{8\beta t_nc^2(1+2M_2)}{1+32M_1}\Big(1+o(\frac{1}{\log ^2n})\Big)\\
  &\geq \frac{4\theta+\mu-8}{4\theta}t_n^{\frac{\mu}{2}-4}\int_{B_{\frac{1}{n}}(0)}\int_{B_{\frac{1}{n}}(0)}\frac{F(t_n\omega_n(y))f(t_n\omega_n (x)) t_n\omega_n(x)}{|x-y|^\mu}dxdy\\
  &\geq \frac{(4\theta+\mu-8)(\varrho-\varepsilon)^2}{256\theta \pi^2}t_n^{\frac{\mu}{2}-4}e^{\frac{1024 c^2 t_n^2 \log^2 n(1+o(\frac{1}{\log^3 n}))}{1+32M_1}}\int_{B_{\frac{1}{n}}(0)}\int_{B_{\frac{1}{n}}(0)}\frac{dxdy}{|x-y|^\mu}.
\end{align*}
Since $B_{\frac{1}{n}-|x|}(0)\subset B_{\frac{1}{n}}(x)$ for any $|x|\leq \frac{1}{n}$, the last integral can be estimated as follows
\begin{align*}
  \int_{B_{\frac{1}{n}}(0)}\int_{B_{\frac{1}{n}}(0)}\frac{dxdy}{|x-y|^\mu}&=\int_{B_{\frac{1}{n}}(0)}dx\int_{B_{\frac{1}{n}}(x)}\frac{dz}{|z|^\mu}\\
  &\geq\int_{B_{\frac{1}{n}}(0)}dx\int_{B_{\frac{1}{n}-|x|}(0)}\frac{dz}{|z|^\mu}\\
  &=\frac{2\pi^2}{4-\mu} \int_{B_{\frac{1}{n}}(0)}\Big(\frac{1}{n}-|x|\Big)^{4-\mu}dx\\
  &=\frac{4\pi^4}{4-\mu} \int_{0}^{\frac{1}{n}}\Big(\frac{1}{n}-r\Big)^{4-\mu}r^3dr\\
  &=\frac{24\pi^4}{(4-\mu)(5-\mu)(6-\mu)(7-\mu)(8-\mu)n^{8-\mu}}=\frac{C_\mu}{n^{8-\mu}},
\end{align*}
where
\begin{equation*}
  C_\mu=\frac{24\pi^4}{(4-\mu)(5-\mu)(6-\mu)(7-\mu)(8-\mu)}.
\end{equation*}
Consequently, we obtain
\begin{align}\label{tomainuse}
  &\frac{128 c^2t_n^2}{1+32M_1}\Big(\frac{4+M_3}{4}+\log n+o(\frac{1}{\log ^2n})\Big)+\frac{8\beta t_nc^2(1+2M_2)}{1+32M_1}\Big(1+o(\frac{1}{\log ^2n})\Big)\nonumber\\
  &\geq \frac{(4\theta+\mu-8)C_\mu(\varrho-\varepsilon)^2}{256\theta \pi^2}t_n^{\frac{\mu}{2}-4}e^{\Big(\frac{1024 c^2 t_n^2 \log n(1+o(\frac{1}{\log^3 n}))}{1+32M_1}-(8-\mu)\Big)\log n},
\end{align}
which is a contradiction.
 Thus $l\in(0,\infty)$, and $\lim\limits_{n\rightarrow\infty}t_n=0$, $\lim\limits_{n\rightarrow\infty}(t_n\log n)=\infty$. By \eqref{tomainuse}, letting $n\rightarrow\infty$, we have that
\begin{equation*}
 0< l\leq \frac{(1+32M_1)(8-\mu)}{1024c^2}.
\end{equation*}
If $0<l< \frac{(1+32M_1)(8-\mu)}{1024c^2}$, then
\begin{align*}
  \lim\limits_{n\rightarrow\infty}g_n(t_n)&\leq  \lim\limits_{n\rightarrow\infty}\Big(\frac{ t_n^2}{2}\int_{\mathbb{R}^4}|\Delta \omega_n|^2dx+\frac{\beta t_n}{2}\int_{\mathbb{R}^4}|\nabla \omega_n|^2dx\Big)=\frac{64 c^2 l }{1+32M_1}<\frac{8-\mu}{16}.
\end{align*}
If $l= \frac{(1+32M_1)(8-\mu)}{1024c^2}$, by \eqref{deomega}, \eqref{ftF} and $\frac{3}{2}-\frac{\mu}{8}>1$, we derive that for $n\in \mathbb{N}^+$ large enough,
\begin{align*}
  \frac{t_n^{\frac{\mu}{2}-4}}{2}\int_{\mathbb{R}^4}(I_\mu*F(t_n\omega_n))F(t_n\omega_n)dx&\geq \frac{(\varrho-\varepsilon)^2}{8192\pi^4t_n^{6-\frac{\mu}{2}}} \int_{B_{\frac{1}{n}}(0)}\int_{B_{\frac{1}{n}}(0)} \frac{e^{32\pi^2 t_n^2\omega_n^2(y)}e^{32\pi^2  t_n^2\omega_n^2(x)}}{\omega_n(y)|x-y|^\mu\omega_n(x)}dxdy\\
  &\geq\frac{1024^{1-\frac{\mu}{4}}C_\mu c^{4-\frac{\mu}{2}}(8\pi^2)^{\frac{3}{2}-\frac{\mu}{8}}(\varrho-\varepsilon)^2}{\pi^4(1+32M_1)^{2-\frac{\mu}{4}}(8-\mu)^{3-\frac{\mu}{4}}}n^{\frac{1024 c^2 t_n^2 \log n(1+o(\frac{1}{\log^3 n}))}{1+32M_1}-(8-\mu)}  .
\end{align*}
Hence, we obtain for $n\in \mathbb{N}^+$ large enough,
\begin{align*}
  g_n(t_n)&\leq\frac{64 c^2t_n^2\log n}{1+32M_1}+\frac{(4+M_3)(8-\mu)}{64\log n}+\frac{8\beta c^2(1+2M_2)}{1+32M_1}\sqrt{\frac{(1+32M_1)(8-\mu)}{1024c^2 \log n}}\\
  &\quad-\frac{1024^{1-\frac{\mu}{4}}C_\mu c^{4-\frac{\mu}{2}}(8\pi^2)^{\frac{3}{2}-\frac{\mu}{8}}(\varrho-\varepsilon)^2}{\pi^4(1+32M_1)^{2-\frac{\mu}{4}}(8-\mu)^{3-\frac{\mu}{4}}}n^{\frac{1024 c^2 t_n^2 \log n(1+o(\frac{1}{\log^3 n}))}{1+32M_1}-(8-\mu)}.
\end{align*}
For any fixed $n\in \mathbb{N}^+$ large enough, let
\begin{align*}
  k_n(t):&=\frac{64 c^2t^2}{1+32M_1}+\frac{(4+M_3)(8-\mu)}{64\log n}+\frac{8\beta c^2(1+2M_2)}{1+32M_1}\sqrt{\frac{(1+32M_1)(8-\mu)}{1024c^2 \log n}}\\&\quad-\frac{1024^{1-\frac{\mu}{4}}C_\mu c^{4-\frac{\mu}{2}}(8\pi^2)^{\frac{3}{2}-\frac{\mu}{8}}(\varrho-\varepsilon)^2}{\pi^4(1+32M_1)^{2-\frac{\mu}{4}}(8-\mu)^{3-\frac{\mu}{4}}}n^{\frac{1024 c^2 t^2(1+o(\frac{1}{\log^3 n}))}{1+32M_1}-(8-\mu)}
\end{align*}
Then $g_n(t_n)\leq \max\limits_{t\geq0}k_n(t)$. A direct computation shows that there exists $t_n^*>0$ such that $\max\limits_{t\geq0}k_n(t)=k_n(t_n^*)$ and $k_n'(t_n^*)=0$, from which we get
\begin{equation*}
  \frac{128 c^2}{1+32M_1}=\frac{1024^{1-\frac{\mu}{4}}C_\mu c^{4-\frac{\mu}{2}}(8\pi^2)^{\frac{3}{2}-\frac{\mu}{8}}(\varrho-\varepsilon)^2}{\pi^4(1+32M_1)^{2-\frac{\mu}{4}}(8-\mu)^{3-\frac{\mu}{4}}}n^{\frac{1024 c^2 (t_n^*) ^2(1+o(\frac{1}{\log^3 n}))}{1+32M_1}-(8-\mu)}\frac{2048c^2 \log n(1+o(\frac{1}{\log^3 n}))}{1+32M_1}.
\end{equation*}
Hence,
\begin{equation*}
  g_n(t_n)\leq \frac{64 c^2(t_n^*)^2}{1+32M_1}-\frac{1}{16\log n(1+o(\frac{1}{\log^3 n}))}+\frac{(4+M_3)(8-\mu)}{64\log n}+\frac{8\beta c^2(1+2M_2)}{1+32M_1}\sqrt{\frac{(1+32M_1)(8-\mu)}{1024c^2 \log n}}.
\end{equation*}
We claim that $\lim\limits_{n\rightarrow\infty}(t_n^*)^2<\frac{(1+32M_1)(8-\mu)}{1024c^2}$, otherwise, we have
\begin{equation*}
 \lim\limits_{n\rightarrow\infty} \frac{1024^{1-\frac{\mu}{4}}C_\mu c^{4-\frac{\mu}{2}}(8\pi^2)^{\frac{3}{2}-\frac{\mu}{8}}(\varrho-\varepsilon)^2}{\pi^4(1+32M_1)^{2-\frac{\mu}{4}}(8-\mu)^{3-\frac{\mu}{4}}}n^{\frac{1024 c^2 (t_n^*) ^2(1+o(\frac{1}{\log^3 n}))}{1+32M_1}-(8-\mu)}\frac{2048c^2 \log n(1+o(\frac{1}{\log^3 n}))}{1+32M_1}=\infty,
\end{equation*}
which is a contradiction. Thus
\begin{align*}
  \lim\limits _{n\rightarrow\infty} g_n(t_n)&\leq  \lim\limits _{n\rightarrow\infty} \Big[\frac{64 c^2(t_n^*)^2}{1+32M_1}-\frac{1}{16\log n(1+o(\frac{1}{\log^3 n}))}+\frac{(4+M_3)(8-\mu)}{64\log n}\\&\qquad\qquad\quad+\frac{8\beta c^2(1+2M_2)}{1+32M_1}\sqrt{\frac{(1+32M_1)(8-\mu)}{1024c^2 \log n}}\Big]<\frac{8-\mu}{16}.
\end{align*}
This ends the proof.
\end{proof}

\section{The monotonicity of the function $c\mapsto E(c)$}\label{Beha}

To guarantee the weak limit of a $(PS)_{E(c)}$ sequence is a ground state solution of \eqref{abs1}-\eqref{abs2}, in this section, we study the monotonicity of the function $c\mapsto E(c)$.

\begin{lemma}\label{nonincreasing}
Assume that $f$ satisfies $(f_1)-(f_3)$, if $(f_6)$ or $(f_7)$ holds, then the function $c\mapsto E(c)$ is non-increasing on $(0,+\infty)$.
\end{lemma}
\begin{proof}
For any given $c>0$, if $\tilde{c}>c$, we prove that $E(\tilde{c})\leq E(c)$. By the definition of $E(c)$, for any $\varepsilon>0$, there exists $u\in \mathcal{P}(c)$ such that
 \begin{equation}\label{nonin0}
 \mathcal J(u)\leq E(c)+\frac{\varepsilon}{3}.
\end{equation}
Consider a cut-off function $\phi\in C_0^\infty(\mathbb{R}^4,[0,1])$ such that $\phi(x)=1$ if $|x|\leq1$, $\phi(x)=0$ if $|x|\geq2$. For any small $\delta>0$, define $u_\delta(x)=\phi(\delta x)u(x)\in H^2(\mathbb{R}^4)\backslash\{0\}$, then $u_\delta\rightarrow u$ in $H^2(\mathbb{R}^4)$ as $\delta\rightarrow0^+$. By Lemmas \ref{biancon} and \ref{equi}, we have $s_{u_\delta}\rightarrow s_u=0$ in $\mathbb{R}$ as $\delta\rightarrow0^+$ and
\begin{equation*}
  \mathcal{H}(u_\delta,s_{u_\delta})\rightarrow \mathcal{H}(u,s_{u})=u \quad\text{in $H^2(\mathbb{R}^4)$ as $\delta\rightarrow0^+$}.
\end{equation*}
Fix $\delta_0>0$ small enough such that
\begin{equation}\label{nonin1}
  \mathcal J(\mathcal{H}(u_{\delta_0},s_{u_{\delta_0}}))\leq \mathcal J(u)+\frac{\varepsilon}{3}.
\end{equation}
Let $v\in C_0^\infty(\mathbb{R}^4)$ satisfy $supp(v)\subset B_{1+\frac{4}{\delta_0}}(0)\backslash B_{\frac{4}{\delta_0}}(0)$, and set
\begin{equation*}
  v_{\delta_0}=\frac{\tilde{c}^2-\|u_{\delta_0}\|_2^2}{\|v\|_2^2}v.
\end{equation*}
Define $\omega_\lambda:=u_{\delta_0}+\mathcal{H}(v_{\delta_0},\lambda)$ for any $\lambda<0$. Since $$dist(u_{\delta_0},\mathcal{H}(v_{\delta_0},\lambda))\geq \frac{2}{\delta_0}>0,$$ we have $\|\omega_\lambda\|_2^2=\tilde{c}^2$, i.e., $\omega_\lambda \in S(\tilde{c})$.

We claim that $s_{\omega_\lambda}$ is bounded from above as $\lambda\rightarrow-\infty$. Otherwise, by \eqref{imp}, $(f_3)$ and $\omega_\lambda\rightarrow u_{\delta_0}\neq0$ a.e. in $\mathbb{R}^4$ as $\lambda\rightarrow-\infty$,
\begin{align*}
  0&\leq \lim\limits_{n\rightarrow\infty}e^{-4s_{\omega_\lambda}}\mathcal J(\mathcal{H}(\omega_\lambda,s_{\omega_\lambda}))\\
  &\leq \lim\limits_{n\rightarrow\infty}\frac{1}{2}\Big[\int_{\mathbb{R}^4}|\Delta \omega_\lambda|^2dx+\frac{\beta}{e^{2s_{\omega_\lambda}}}\int_{\mathbb{R}^4}|\nabla \omega_\lambda|^2dx-e^{(4\theta+\mu-12)s_{\omega_\lambda}}\int_{\mathbb{R}^4}(I_\mu*F(\omega_\lambda))F(\omega_\lambda)dx\Big]=-\infty,
\end{align*}
which is a contradiction. Thus $s_{\omega_\lambda}+\lambda\rightarrow-\infty$, by $(f_3)$, we get
\begin{equation}\label{nonin2}
  \mathcal J(\mathcal{H}(v_{\delta_0},s_{\omega_\lambda}+\lambda))\leq \frac{e^{4(s_{\omega_\lambda}+\lambda)}}{2}\|\Delta v_{\delta_0}\|_2^2+ \frac{e^{2(s_{\omega_\lambda}+\lambda)}}{2}\|\nabla v_{\delta_0}\|_2^2\rightarrow0, \quad\text{as}\,\,\, \lambda\rightarrow-\infty.
\end{equation}
Now, using Lemma \ref{equi}, $\eqref{nonin0}-\eqref{nonin2}$, we obtain
\begin{align*}
  E(\tilde{c})\leq \mathcal J(\mathcal{H}(\omega_\lambda,s_{\omega_\lambda}))&=\mathcal J(\mathcal{H}(u_{\delta_0},s_{\omega_\lambda}))+\mathcal J(\mathcal{H}(\mathcal{H}(v_{\delta_0},\lambda),s_{\omega_\lambda}))\\
  &=\mathcal J(\mathcal{H}(u_{\delta_0},s_{\omega_\lambda}))+\mathcal J(\mathcal{H}(v_{\delta_0},s_{\omega_\lambda}+\lambda))\\
  &\leq\mathcal J(\mathcal{H}(u_{\delta_0},s_{u_{\delta_0}}))+\mathcal J(\mathcal{H}(v_{\delta_0},s_{\omega_\lambda}+\lambda))\leq E(c)+\varepsilon.
\end{align*}
By the arbitrariness of $\varepsilon>0$, we deduce that $E(\tilde{c})\leq E(c)$ for any $\tilde{c}>c$.
\end{proof}
\begin{lemma}\label{close to}
Assume that $f$ satisfies $(f_1)-(f_3)$, $(f_6)$, or $(f_1)-(f_3)$, $(f_7)$. Suppose that \eqref{abs1}-\eqref{abs2} possesses a ground state solution $u$ with $\lambda<0$,
then $E(c')<E(c)$ for any $c'>c$ close to $c$.
\end{lemma}
\begin{proof}
For any $t>0$ and $s\in \mathbb{R}$, we know $\mathcal{H}(tu,s)\in S(tc)$ and
\begin{align*}
\mathcal J(\mathcal{H}(tu,s))=\frac{t^2 e^{4s}}{2}\int_{\mathbb{R}^4}|\Delta u|^2dx+\frac{\beta t^2e^{2s}}{2}\int_{\mathbb{R}^4}|\nabla u|^2dx-\frac{e^{(\mu-8)s}}{2}\int_{\mathbb{R}^4}(I_\mu*F(te^{2s}u))F(te^{2s}u)dx.
\end{align*}
Denote $\alpha(t,s):=\mathcal J(\mathcal{H}(tu,s))$, then
\begin{align*}
  \frac{\partial \alpha(t,s) }{\partial t}&= t e^{4s}\int_{\mathbb{R}^4}|\Delta u|^2dx+\beta te^{2s}\int_{\mathbb{R}^4}|\nabla u|^2dx-e^{(\mu-8)s}\int_{\mathbb{R}^4}(I_\mu*F(te^{2s}u))f(te^{2s}u)e^{2s}udx\\
  &=\frac{\langle\mathcal{J}'(\mathcal{H}(tu,s)),\mathcal{H}(tu,s)\rangle}{t}.
\end{align*}
 By Lemma \ref{biancon}, $\mathcal{H}(tu,s)\rightarrow u$ in $H^2(\mathbb{R}^4)$ as $(t,s)\rightarrow (1,0)$. Since $\lambda<0$, $\langle\mathcal{J}'(u),u\rangle=\lambda \|u\|_2^2=\lambda c^2<0$.
Hence, one can fix $\delta>0$ small enough such that
\begin{equation*}
   \frac{\partial \alpha(t,s) }{\partial t}<0\quad\text{for any $(t,s)\in (1,1+\delta]\times [-\delta,\delta]$}.
\end{equation*}
For any $t\in (1,1+\delta]$ and $s\in [-\delta,\delta]$, using the mean value theorem, we obtain
\begin{equation*}
  \alpha(t,s)=\alpha(1,s)+(t-1)\cdot \frac{\partial \alpha(t,s) }{\partial t}\Big|_{t=\xi}<\alpha(1,s)
\end{equation*}
for some $\xi\in (1,t)$. By Lemma \ref{equi}, $s_{tu}\rightarrow s_u=0$ in $\mathbb{R}$ as $t \rightarrow1^+$. For any $c'>c$ close to $c$, let $t_0=\frac{c'}{c}$, then
\begin{equation*}
  t_0\in (1,1+\delta]\quad \text{and}\quad s_{t_0u}\in [-\delta,\delta]
\end{equation*}
and thus, using Lemma \ref{equi} again,
\begin{equation*}
  E(c')\leq \alpha(t_0,s_{t_0u})<\alpha(1,s_{t_0u})=\mathcal J(\mathcal{H}(u,s_{t_0u}))\leq \mathcal J(u)=E(c).
\end{equation*}
\end{proof}

From Lemmas \ref{nonincreasing} and \ref{close to}, we can directly obtain the following result.
\begin{lemma}\label{conclusion}
Assume that $f$ satisfies $(f_1)-(f_3)$, $(f_6)$, or $(f_1)-(f_3)$, $(f_7)$. Suppose that \eqref{abs1}-\eqref{abs2} possesses a ground state solution with $\lambda<0$,
then $E(c')<E(c)$ for any $c'>c$.
\end{lemma}

\section{Palais-Smale sequence}\label{mini}

In this section, using a minimax principle based on the homotopy stable family of compact subsets of $S(c)$, we will construct a $(PS)_{E(c)}$ sequence on $\mathcal{P}(c)$ for $\mathcal J$.
\begin{proposition}\label{pssequencehomo}
There exists a non-negative $(PS)_{E(c)}$ sequence $\{u_n\}\subset \mathcal{P}(c)$ for $\mathcal J$.
\end{proposition}
Following  by \cite{Wil}, we recall that for any $c>0$, the tangent space of $S(c)$ at $u$ is defined by
\begin{equation*}
  T_u=\Big\{v\in H^2(\mathbb{R}^4):\int_{\mathbb{R}^4}uvdx=0\Big\}.
\end{equation*}
To prove Proposition \ref{pssequencehomo}, we borrow some arguments from \cite{BS1,BS2} and consider the functional $\mathcal{I}:S(c)\rightarrow\mathbb{R}$ defined by
\begin{equation*}
 \mathcal{I}(u)=\mathcal J(\mathcal{H}(u,s_u)),
\end{equation*}
where $s_u\in \mathbb{R}$ is the unique number obtained in Lemma \ref{equi} for any $u\in S(c)$. By Lemma \ref{equi}, we know that $s_u$ is continuous as a mapping of $u\in S(c)$. However, it remains unknown that whether $s_u$ is of class $C^1$. Inspired by \cite[Proposition 2.9]{SW}, we observe that
\begin{lemma}\label{C^1}
Assume that $f$ satisfies $(f_1)-(f_3)$, if $(f_6)$ or $(f_7)$ holds, then the functional $\mathcal{I}:S(c)\rightarrow\mathbb{R}$ is of class $C^1$ and
\begin{align*}
  \langle\mathcal{I}'(u),v\rangle&=e^{4s_u}\int_{\mathbb{R}^4}\Delta u \Delta v dx+\beta e^{2s_u}\int_{\mathbb{R}^4}\nabla u \cdot\nabla v dx-e^{(\mu-8)s_u}\int_{\mathbb{R}^4}(I_\mu*F(e^{2s_u}u))f(e^{2s_u}u)e^{2s_u}v dx\\
  &=\langle\mathcal J'(\mathcal{H}(u,s_u)),\mathcal{H}(v,s_u)\rangle
\end{align*}
for any $u\in S(c)$ and $v\in T_u$.
\end{lemma}
\begin{proof}
Let $u\in S(c)$ and $v\in T_u$, then for any $|t|$ small enough, by Lemma \ref{equi},
\begin{align*}
 \mathcal{I}(u+tv)-\mathcal{I}(u)=&\mathcal J(\mathcal{H}(u+tv,s_{u+tv}))-\mathcal J(\mathcal{H}(u,s_u))\\
  \leq& \mathcal J(\mathcal{H}(u+tv,s_{u+tv}))-\mathcal J(\mathcal{H}(u,s_{u+tv}))\\
  =&\frac{e^{4s_{u+tv}}}{2}\int_{\mathbb{R}^4}\Big[|\Delta (u+tv)|^2-|\Delta u|^2\Big]dx+\frac{\beta e^{2s_{u+tv}}}{2}\int_{\mathbb{R}^4}\Big[|\nabla (u+tv)|^2-|\nabla u|^2\Big]dx\\&-\frac{e^{(\mu-8)s_{u+tv}}}{2}\int_{\mathbb{R}^4}\Big[(I_\mu*F(e^{2s_{u+tv}}(u+tv)))F(e^{2s_{u+tv}}(u+tv))\\
  &-(I_\mu*F(e^{2s_{u+tv}}u))F(e^{2s_{u+tv}}u)\Big] dx\\
  =&\frac{e^{4s_{u+tv}}}{2}\int_{\mathbb{R}^4}\Big(t^2|\Delta v|^2+2t\Delta u \Delta v\Big)dx+\frac{\beta e^{2s_{u+tv}}}{2}\int_{\mathbb{R}^4}\Big(t^2|\nabla v|^2+2t\nabla u \cdot \nabla v\Big)dx\\&-\frac{e^{(\mu-8)s_{u+tv}}}{2}\int_{\mathbb{R}^4}(I_\mu*F(e^{2s_{u+tv}}(u+tv)))f(e^{2s_{u+tv}}(u+\xi_ttv))e^{2s_{u+tv}}tv dx\\&-\frac{e^{(\mu-8)s_{u+tv}}}{2}\int_{\mathbb{R}^4}(I_\mu*F(e^{2s_{u+tv}}u))f(e^{2s_{u+tv}}(u+\xi_ttv))e^{2s_{u+tv}}tv dx,
\end{align*}
where $\xi_t\in (0,1)$. Analogously, we have
\begin{align*}
  \mathcal{I}(u+tv)-\mathcal{I}(u)
  \geq &
  \frac{e^{4s_{u}}}{2}\int_{\mathbb{R}^4}\Big(t^2|\Delta v|^2+2t\Delta u\Delta v\Big)dx+\frac{\beta e^{2s_{u}}}{2}\int_{\mathbb{R}^4}\Big(t^2|\nabla v|^2+2t\nabla u\cdot \nabla v\Big)dx\\&-\frac{e^{(\mu-8)s_{u}}}{2}\int_{\mathbb{R}^4}(I_\mu*F(e^{2s_{u}}(u+tv)))f(e^{2s_{u}}(u+\zeta_ttv))e^{2s_{u}}tv dx\\&-\frac{e^{(\mu-8)s_{u}}}{2}\int_{\mathbb{R}^4}(I_\mu*F(e^{2s_{u}}u))f(e^{2s_{u}}(u+\zeta_ttv))e^{2s_{u}}tv dx,
\end{align*}
where $\zeta_t\in (0,1)$. By Lemma \ref{equi}, $\lim\limits_{t\rightarrow0}s_{u+tv}=s_u$, from the above inequalities, we conclude
\begin{align*}
  \lim\limits_{t\rightarrow0}\frac{\mathcal{I}(u+tv)-\mathcal{I}(u)}{t}=&e^{4s_u}\int_{\mathbb{R}^4}\Delta u \Delta v dx+\beta e^{2s_u}\int_{\mathbb{R}^4}\nabla u \cdot\nabla v dx\\
  &-e^{(\mu-8)s_u}\int_{\mathbb{R}^4}(I_\mu*F(e^{2s_u}u))f(e^{2s_u}u)e^{2s_u}v dx.
\end{align*}
Using Lemma \ref{equi}, we find that the G\^{a}teaux derivative of $\mathcal{I}$ is continuous linear in $v$ and continuous in $u$. Therefore, by \cite[Proposition 1.3]{Wil}, we get $\mathcal{I}$ is of class $C^1$. Changing variables in the integrals, we prove the rest.
\end{proof}

\begin{lemma}\label{pssequencehomo2}
Assume that $f$ satisfies $(f_1)-(f_3)$, if $(f_6)$ or $(f_7)$ holds, let $\mathcal{F}$  be a homotopy stable family of compact subsets of $S(c)$ without boundary (i.e., $B=\emptyset$) and set
\begin{equation*}
  E_{\mathcal{F}}:=\inf\limits_{A\in \mathcal{F}}\max\limits_{u\in A}\mathcal I(u).
\end{equation*}
If $E_{\mathcal{F}}>0$, then there exists a non-negative $(PS)_{E_{\mathcal{F}}}$ sequence $\{u_n\}\subset \mathcal{P}(c)$ for $\mathcal J$.
\end{lemma}
\begin{proof}
Let $\{A_n\}\subset \mathcal{F}$ be a minimizing sequence of $E_{\mathcal{F}}$. We define the mapping
\begin{equation*}
  \eta:[0,1]\times S(c)\rightarrow S(c), \quad\eta(t,u)=\mathcal{H}(u,ts_u).
\end{equation*}
By Lemma \ref{equi}, $\eta(t,u)$ is continuous in $[0,1]\times S(c)$, and satisfies $\eta(t,u)=u$ for all $(t,u)\in \{0\}\times S(c)$. Thus by the definition of $\mathcal{F}$, one has
\begin{equation*}
  D_n:=\eta(1,A_n)=\{\mathcal{H}(u,s_u):u\in A_n\}\subset \mathcal{F}.
\end{equation*}
Obviously, $D_n\subset \mathcal{P}(c)$ for any $n\in \mathbb{N}^+$. Since $\mathcal{I}(\mathcal{H}(u,s))=\mathcal{I}(u)$ for any $u\in S(c)$ and $s\in \mathbb{R}$, then
\begin{equation*}
  \max\limits_{u\in D_n}\mathcal{I}(u)=\max\limits_{u\in A_n}\mathcal{I}(u)\rightarrow E_{\mathcal{F}},\quad\text{as $n\rightarrow\infty$,}
\end{equation*}
which implies that $\{D_n\}\subset \mathcal{F}$ is another minimizing sequence of $E_{\mathcal{F}}$. By Lemma \ref{Ghouss}, we obtain a $(PS)_{E_{\mathcal{F}}}$ sequence $\{v_n\}\subset \mathcal{S}(c)$ for $\mathcal I$ such that $\lim\limits_{n\rightarrow\infty}dist(v_n,D_n)=0$. By Lemma \ref{equi}, we know $s_{-u}=s_u$ for any $u\in S(c)$, thus $\mathcal{I}(u)$ is even in $S(c)$. Without loss of generality, we assume that $v_n\geq0$ for all $n\in \mathbb{N}^+$.
Let
\begin{equation*}
  u_n:=\mathcal{H}(v_n,s_{v_n}),
\end{equation*}
then $u_n\geq0$ for all $n\in \mathbb{N}^+$, and we prove that $\{u_n\}\subset \mathcal{P}(c)$ is the desired sequence.

We claim that there exists $C>0$ such that $e^{-4s_{v_n}}\leq C$ for any $n\in \mathbb{N}^+$. Indeed, we observe that
\begin{equation*}
  e^{-4s_{v_n}}=\frac{\int_{\mathbb{R}^4}|\Delta v_n|^2dx}{\int_{\mathbb{R}^4}|\Delta u_n|^2dx}.
\end{equation*}
Since $\{u_n\}\subset \mathcal{P}(c)$, by Lemma \ref{inf}, there exists a constant $C_1>0$ such that $\int_{\mathbb{R}^4}|\Delta u_n|^2dx\geq C_1$ for any $n\in \mathbb{N}^+$. Regarding the term of $\{v_n\}$, since $D_n\subset \mathcal{P}(c)$ for any $n\in \mathbb{N}^+$ and for any $u\in \mathcal{P}(c)$, one has $\mathcal{J}(u)=\mathcal{I}(u)$, thus
\begin{equation*}
   \max\limits_{u\in D_n}\mathcal{J}(u)= \max\limits_{u\in D_n}\mathcal{I}(u)\rightarrow E_{\mathcal{F}},\quad\text{as $n\rightarrow\infty$.}
\end{equation*}
This fact together with $D_n\subset \mathcal{P}(c)$ and $(f_3)$ yields that $\{D_n\}$ is uniformly bounded in $H^2(\mathbb{R}^4)$, thus from $\lim\limits_{n\rightarrow\infty}dist(v_n,D_n)=0$, we obtain $\sup\limits_{n\geq1}\|v_n\|^2<\infty$. This proves the claim.

From $\{u_n\}\subset \mathcal{P}(c)$, it follows that
\begin{equation*}
  \mathcal{J}(u_n)=\mathcal{I}(u_n)=\mathcal{I}(v_n)\rightarrow E_{\mathcal{F}},\quad\text{as $n\rightarrow\infty$.}
\end{equation*}
For any $\varphi \in T_{u_n}$, we have
\begin{equation*}
  \int_{\mathbb{R}^4}v_n\mathcal{H}(\varphi,-s_{v_n})dx=\int_{\mathbb{R}^4}\mathcal{H}(v_n,s_{v_n})\varphi dx=\int_{\mathbb{R}^4}u_n\varphi dx=0,
\end{equation*}
which means $\mathcal{H}(\varphi,-s_{v_n})\in T_{v_n}$. Also, by the claim, we have
\begin{equation*}
  \|\mathcal{H}(\varphi,-s_{v_n})\|^2=e^{-4s_{v_n}}\|\Delta \varphi\|_2^2+2e^{-2s_{v_n}}\|\nabla \varphi\|_2^2+\|\varphi\|_2^2\leq C\|\Delta \varphi\|_2^2+2\sqrt{C}\|\nabla \varphi\|_2^2+\|\varphi\|_2^2\leq \max\{1,C\}\|\varphi\|^2.
\end{equation*}
By Lemma \ref{C^1}, for any $\varphi\in T_{u_n}$, we deduce that
\begin{align*}
  |\langle\mathcal{J}'(u_n),\varphi\rangle|&=|\langle\mathcal{J}'(\mathcal{H}(v_n,s_{v_n}),\mathcal{H}(\mathcal{H}(\varphi,-s_{v_n}),s_{v_n})\rangle|=|\langle\mathcal{I}'(v_n),\mathcal{H}(\varphi,-s_{v_n})\rangle|\\
  &\leq \|\mathcal{I}'(v_n)\|_*\cdot \|\mathcal{H}(\varphi,-s_{v_n})\|\leq \max\{1,\sqrt{C}\}\|\mathcal{I}'(v_n)\|_*\cdot  \|\varphi\|,
\end{align*}
which implies that $\|\mathcal{J}'(u_n)\|_*\leq \max\{1,\sqrt{C}\}\|\mathcal{I}'(v_n)\|_*\rightarrow0$ as $n\rightarrow\infty$. Thus $\{u_n\}\subset \mathcal{P}(c)$ is a non-negative $(PS)_{E_{\mathcal{F}}}$ sequence for $\mathcal J$.
\end{proof}
\noindent
{\bf Proof of Proposition \ref{pssequencehomo}:} Note that the class $\mathcal{F}$ of all singletons included in  $S(c)$ is a homotopy stable family of compact subsets of $S(c)$ without boundary. By Lemma \ref{pssequencehomo2}, if $E_{\mathcal{F}}>0$, then there exists a non-negative $(PS)_{E_{\mathcal{F}}}$ sequence $\{u_n\}\subset \mathcal{P}(c)$ for $\mathcal J$. By Lemma \ref{inf}, we know $E(c)>0$, so if we can prove that $E_{\mathcal{F}}=E(c)$, then we complete the proof.
In fact, by the definition of $\mathcal{F}$, we have
\begin{equation*}
  E_{\mathcal{F}}=\inf\limits_{A\in \mathcal{F}}\max\limits_{u\in A}\mathcal I(u)=\inf\limits_{u\in S(c)}\mathcal I(u)=\inf\limits_{u\in S(c)}\mathcal I(\mathcal{H}(u,s_u))=\inf\limits_{u\in S(c)}\mathcal J(\mathcal{H}(u,s_u)).
\end{equation*}
For any $u\in S(c)$, it follows from $\mathcal{H}(u,s_u)\in \mathcal{P}(c)$ that $\mathcal J(\mathcal{H}(u,s_u))\geq E(c)$, by the arbitrariness of $u\in S(c)$, we get $E_{\mathcal{F}}\geq E(c)$. On the other hand, for any $u\in \mathcal{P}(c)$, by Lemma \ref{equi}, we deduce that $s_u=0$ and $\mathcal{J}(u)=\mathcal J(\mathcal{H}(u,0))\geq \inf\limits_{u\in S(c)}\mathcal J(\mathcal{H}(u,s_u))$, by the arbitrariness of $u\in \mathcal{P}(c)$, we get $E(c)\geq E_{\mathcal{F}}$.
\qed

For the sequence $\{u_n\}$ obtained in Proposition \ref{pssequencehomo}, by $(f_3)$, we know that $\{u_n\}$ is bounded in $H^2(\mathbb{R}^4)$, up to a subsequence, we assume that $u_n\rightharpoonup u_c$ in $H^2(\mathbb{R}^4)$.
Furthermore, by ${\mathcal J}'|_{S(c)}(u_n)\rightarrow0$ as $n\rightarrow\infty$ and Lagrange multiplier rule, there exists $\{\lambda_n\}\subset \mathbb{R}$ such that
\begin{equation}\label{lagrange}
  \Delta^2u_n-\beta\Delta u_n=\lambda_n u_n+(I_\mu*F(u_n))f(u_n)+o_n(1).
\end{equation}

\begin{lemma}\label{nonzero}
Assume that $f$ satisfies $(f_1)-(f_5)$, if $(f_6)$ or $(f_7)$ holds, then up to a subsequence and up to translations in $\mathbb{R}^4$, $u_n\rightharpoonup u_c\neq0$ in $H^2(\mathbb{R}^4)$.
\end{lemma}
\begin{proof}
We claim that
\begin{equation*}
  \kappa:=\limsup\limits_{n\rightarrow\infty}\Big(\sup\limits_{y\in \mathbb{R}^4}\int_{B(y,r)}|u_n|^2dx\Big)>0.
\end{equation*}
Otherwise, by the Lions lemma \cite[Lemma 1.21]{Wil}, $u_n\rightarrow0$ in $L^p(\mathbb{R}^4)$ for any $p>2$.

By ${\mathcal J}(u_n)=E(c)+o_n(1)$, $P(u_n)=0$ and $(f_3)$, we get
\begin{align*}
  {\mathcal J}(u_n)-\frac{1}{4}P(u_n)&=\frac{\beta }{4}\int_{\mathbb{R}^4}|\nabla u_n|^2dx+\frac{\mu-12}{8}\int_{\mathbb{R}^4}(I_\mu*F(u_n))F(u_n)dx+\frac{1}{2}\int_{\mathbb{R}^4}(I_\mu*F(u_n))f(u_n)u_n dx\nonumber\\
  &\geq\frac{\beta }{4}\int_{\mathbb{R}^4}|\nabla u_n|^2dx+\frac{4\theta+\mu-12}{8}\int_{\mathbb{R}^4}(I_\mu*F(u_n))F(u_n)dx.
\end{align*}
Since $\beta\geq0$, by $\theta>3-\frac{\mu}{4}$, we have
\begin{equation}\label{FF}
 \limsup\limits_{n\rightarrow\infty} \int_{\mathbb{R}^4} (I_\mu*F(u_n))F(u_n) dx\leq \frac{8 E(c)}{4\theta+\mu-12}.
\end{equation}
This together with $P(u_n)=0$ and the boundedness of $\{u_n\}$ implies that, up to a subsequence, there exists constant $L_0>0$ such that
\begin{equation}\label{L0}
  \int_{\mathbb{R}^4} (I_\mu*F(u_n))f(u_n)u_n dx\leq L_0.
\end{equation}
For any $\delta\in (0,\frac{M_0L_0}{R_0})$, from $(f_3)$, $(f_4)$, we can choose $M_\delta>\frac{M_0L_0}{\delta}>R_0$, then
\begin{equation*}
  \int\limits_{|u_n|\geq M_\delta} (I_\mu*F(u_n))F(u_n)dx\leq M_0\int\limits_{|u_n|\geq M_\delta} (I_\mu*F(u_n))|f(u_n)|dx\leq \frac{M_0}{M_\delta}\int\limits_{|u_n|\geq M_\delta} (I_\mu*F(u_n))f(u_n)u_ndx<\delta
  .
\end{equation*}
On the other hand, by \eqref{fcondition2} and \eqref{HLS},
 \begin{equation*}
  \int\limits_{|u_n|\leq M_\delta} (I_\mu*F(u_n))F(u_n)dx\leq C (\|u_n\|_{\nu+1}^{\nu+1}+\|u_n\|_{q+1}^{q+1})=o_n(1).
\end{equation*}
Due to the arbitrariness of $\delta>0$, we deduce that
\begin{equation*}
  \int\limits_{\mathbb{R}^4} (I_\mu*F(u_n))F(u_n)dx=o_n(1).
\end{equation*}
Hence, by Lemma \ref{control}, we have
\begin{equation*}
 \limsup\limits_{n\rightarrow\infty} \|\Delta u_n\|_2^2\leq 2E(c)<\frac{8-\mu}{8}.
\end{equation*}
Up to a subsequence, we assume that $\sup\limits_{n\in \mathbb{N}^+} \|\Delta u_n\|_2^2<\frac{8-\mu}{8}$. Using \eqref{HLS} again, we have
\begin{equation*}
  \int\limits_{\mathbb{R}^4} (I_\mu*F(u_n))f(u_n)u_ndx\leq C\|F(u_n)\|_{\frac{8}{8-\mu}}\|f(u_n)u_n\|_{\frac{8}{8-\mu}}.
\end{equation*}
Fix $\alpha>32\pi^2$ close to $32\pi^2$ and $m>1$ close to $1$ such that
 \begin{equation*}
   \sup\limits_{n\in \mathbb{N}^+} \frac{8\alpha m \|\Delta u_n\|_2^2}{8-\mu}\leq 32\pi^2.
\end{equation*}
Arguing as \eqref{close1}, for $m'=\frac{m}{m-1}$, we have
\begin{equation*}
  \|F(u_n)\|_{\frac{8}{8-\mu}}\leq C\|u_n\|_{\frac{8(\nu+1)}{8-\mu}}^{\nu+1}+C\|u_n\|_{\frac{8(q+1)m'}{8-\mu}}^{q+1}\rightarrow0,\quad\text{as $n\rightarrow\infty$}.
\end{equation*}
Similarly, we can prove that
\begin{equation*}
  \|f(u_n)u_n\|_{\frac{8}{8-\mu}}\rightarrow0,\quad\text{as $n\rightarrow\infty$}.
\end{equation*}
Therefore, we get
\begin{equation*}
  \int\limits_{\mathbb{R}^4} (I_\mu*F(u_n))f(u_n)u_ndx=o_n(1).
\end{equation*}
Recalling that $P(u_n)=0$, so $\|\Delta u_n\|_2=o_n(1)$, which implies $E(c)=0$, this is impossible, since $E(c)>0$. According to $\kappa>0$, there exists $\{y_n\}\subset \mathbb{R}^4$  such that $\int_{B(y_n,1)}|u_n|^2dx>\frac{\kappa}{2}$, i.e., $\int_{B(0,1)}|u_n(x-y_n)|^2dx>\frac{\kappa}{2}$. Up to a subsequence and up to translations in $\mathbb{R}^4$, $u_n\rightharpoonup u_c\neq0$ in $H^2(\mathbb{R}^4)$.
\end{proof}
\begin{lemma}\label{ne}
Assume that $f$ satisfies $(f_1)-(f_5)$, if $(f_6)$ or $(f_7)$ holds, then $\{\lambda_n\}$ is bounded in $\mathbb{R}$. Furthermore, there exists $\beta_*>0$, such that  $\lambda_n\rightarrow\lambda_c$ with some $\lambda_c<0$ as $n\rightarrow\infty$ for any $\beta\in [0,\beta_*)$.
\end{lemma}
\begin{proof}
Testing \eqref{lagrange} with $u_n$, we have
\begin{equation*}
   \int_{\mathbb{R}^4}|\Delta u_n|^2dx+\beta\int_{\mathbb{R}^4}|\nabla u_n|^2dx=\lambda_n \int_{\mathbb{R}^4}|u_n|^2dx+\int_{\mathbb{R}^4}(I_\mu*F(u_n))f(u_n)u_ndx+o_n(1).
\end{equation*}
Combing this with $P(u_n)=0$ leads to
\begin{equation*}
  \lambda_n c^2=\frac{\beta}{2}\int_{\mathbb{R}^4}|\nabla u_n|^2dx-\frac{8-\mu}{4}\int_{\mathbb{R}^4}(I_\mu*F(u_n))F(u_n)dx+o_n(1).
\end{equation*}
Thus
\begin{equation}\label{lamudabdd}
  \limsup\limits_{n\rightarrow\infty}|\lambda_n| \leq \limsup\limits_{n\rightarrow\infty}\Big(\frac{\beta}{2c^2}\int_{\mathbb{R}^4}|\nabla u_n|^2dx+\frac{8-\mu}{4c^2}\int_{\mathbb{R}^4}(I_\mu*F(u_n))F(u_n)dx\Big),
\end{equation}
and
\begin{equation}\label{lamudane}
   \limsup\limits_{n\rightarrow\infty}\lambda_n = \limsup\limits_{n\rightarrow\infty}\frac{\beta}{2c^2}\int_{\mathbb{R}^4}|\nabla u_n|^2dx- \liminf\limits_{n\rightarrow\infty}\frac{8-\mu}{4c^2}\int_{\mathbb{R}^4}(I_\mu*F(u_n))F(u_n)dx.
\end{equation}
From \eqref{FF} and \eqref{lamudabdd}, it follows that $\{\lambda_n\}$ is bounded in $\mathbb{R}$. Since $\{u_n\}$ is bounded in $H^2(\mathbb{R}^4)$, there exists constant $M>0$ such that $\|u_n\|^2\leq M$. Choosing
\begin{equation*}
\beta_*=\frac{8-\mu}{2M}\int_{\mathbb{R}^4}(I_\mu*F(u_c))F(u_c)dx,
\end{equation*}
then for any $\beta\in [0,\beta_*)$, by \eqref{lamudane}, using Fatou lemma, we have $\limsup\limits_{n\rightarrow\infty}\lambda_n<0$.  Since $\{\lambda_n\}$ is bounded, up to a subsequence, we have that $\lambda_n\rightarrow\lambda_c$ with some $\lambda_c<0$ as $n\rightarrow\infty$ for any $\beta\in [0,\beta_*)$.
\end{proof}

\section{Proof of the main result}\label{main}

\noindent
{\it Proof of Theorem \ref{th2}:}
Under the assumptions of Theorem \ref{th2}, from \eqref{lagrange}, \eqref{L0}, Lemma \ref{strong}, Lemmas \ref{nonzero} and \ref{ne}, we can see that $u_c$ is a positive weak solution of \eqref{abs1} with $\lambda<0$, and $P(u_c)=0$.
Using the Br\'{e}zis-Lieb Lemma \cite[Lemma 1.32]{Wil}, we have
\begin{equation*}
  \int_{\mathbb{R}^4}|u_n|^2dx=\int_{\mathbb{R}^4}|u_n-u_c|^2dx+\int_{\mathbb{R}^4}|u_c|^2dx+o_n(1).
\end{equation*}
Let $c_1:=\|u_c\|_2>0$ and $c_{2,n}:=\|u_n-u_c\|_2$, then $c^2=c_1^2+c_{2.n}^2+o_n(1)$. Since $P(u_c)=0$, using $(f_3)$ and Fatou lemma, we have
\begin{align*}
  \mathcal{J}(u_c)={\mathcal J}(u_c)-\frac{1}{4}P(u_c)&=\frac{\beta}{4}\int_{\mathbb{R}^4}|\nabla u_c|^2dx+\frac{1}{2}\int_{\mathbb{R}^4}(I_\mu*F(u_c))\Big(f(u_c)u_c-(3-\frac{\mu}{4})F(u_c)\Big) dx\nonumber\\
  &\leq\frac{1}{2}\liminf\limits_{n\rightarrow\infty}\Big[\frac{\beta}{2}\int_{\mathbb{R}^4}|\nabla u_n|^2dx+\int_{\mathbb{R}^4}(I_\mu*F(u_n))\Big(f(u_n)u_n-(3-\frac{\mu}{4})F(u_n)\Big) dx\Big]\nonumber\\
  &=\liminf\limits_{n\rightarrow\infty}({\mathcal J}(u_n)-\frac{1}{4}P(u_n))=E(c).
\end{align*}
On the other hand, it follows from Lemma \ref{nonincreasing} that $\mathcal{J}(u_c)\geq E(c_1)\geq E(c)$. Thus $\mathcal{J}(u_c)= E(c_1)= E(c)$. By Lemma \ref{conclusion},  we obtain $c=c_1$. This implies $u_c$ is a positive ground state solution of \eqref{abs1}-\eqref{abs2}.

\subsection*{Acknowledgments}
The authors have been supported by National Natural Science Foundation of China 11971392 and Natural Science Foundation of Chongqing, China cstc2021ycjh-bgzxm0115.

\end{document}